\documentclass[leqno,11pt]{amsart}
\usepackage{amssymb,amsmath,latexsym,amsfonts,amsbsy, amsthm}
\usepackage{color}
\usepackage{graphicx}
\usepackage{enumerate}
\usepackage{hyperref}

\let\pa=\partial
\let\al=\alpha

\let\e=\varepsilon

\let\lam=\lambda

\let\f=\frac

\let\Om=\Omega

\let\ve=\varepsilon
\let\pa=\partial

\def\cN{{\cal N}}

\def\cN{{\mathcal N}}

\newcommand{\beq}{\begin{equation}}
\newcommand{\eeq}{\end{equation}}
\newcommand{\beqo}{\begin{equation*}}
\newcommand{\eeqo}{\end{equation*}}
\newcommand{\ben}{\begin{eqnarray}}
\newcommand{\een}{\end{eqnarray}}
\newcommand{\beno}{\begin{eqnarray*}}
\newcommand{\eeno}{\end{eqnarray*}}

\newtheorem{lem}{Lemma}[section]
\newtheorem{rmk}{Remark}[section]

\newtheorem{case}{Case}{\bf}{}

\newtheorem{theorem}{Theorem}[section]

\newtheorem{lemma}[theorem]{Lemma}
\newtheorem{proposition}[theorem]{Proposition}
\newtheorem{corol}[theorem]{Corollary}

\newtheorem{Theorem}{Theorem}[section]

\newtheorem{Remark}[Theorem]{Remark}

\newcommand{\ud}{\mathrm{d}}

\newcommand{\xx}{\mathbf{x}}

\newcommand{\nn}{\mathbf{n}}
\newcommand{\ee}{\mathbf{e}}
\newcommand{\te}{\tilde{\mathbf{e}}}

\newcommand{\mm}{\mathbf{m}}

\newcommand{\II}{\mathbf{I}}

\newcommand{\CR}{\mathcal{R}}

\newcommand{\CL}{\mathcal{L}}

\newcommand{\BP}{\mathbb{P}}
\newcommand{\BQ}{\mathbb{Q}}

\newcommand{\Qa}{{\mathbb{Q}}}
\newcommand{\Qp}{{\mathbb{Q}_{phy}}}

\newcommand{\BS}{{\mathbb{S}^2}}
\newcommand{\BR}{{\mathbb{R}^2}}

\begin{document}

\title[Stability of 2D point defects with singular potential]{Point defects in 2-D liquid crystals with singular potential: profiles and stability}

\author{Zhiyuan Geng}
\address{Basque Center for Applied Mathematics, Alameda de Mazarredo 14, 48009 Bilbao, Spain}
\email{zgeng@bcamath.org}

\author{Wei Wang}
\address{Department of Mathematics, Zhejiang University, 310027, Hangzhou, P. R. China}
\email{wangw07@zju.edu.cn}

\begin{abstract}
We study radial symmetric point defects with degree $\frac k2$ in 2D disk or $\BR$ in $Q$-tensor framework with singular bulk energy, which is defined by Bingham closure. First, we obtain the existence of solutions for the profiles of radial symmetric point defects with degree $\frac k2$ in 2D disk or $\BR$. Then we prove that the solution is stable for $|k|=1$ and unstable for $|k|>1$. Some identities are derived and used throughout the proof of existence and stability/instability.

\end{abstract}

 \maketitle


\section{Introduction}\label{sec:intro}

\subsection{Backgrounds of the problem}
The liquid crystal state is a distinct phase of matter observed between the crystalline solid and isotropic liquid.
The nematic liquid crystal is characterized by rod-like molecules which exhibit long-range orientational order
but no positional order. In general, there are three main continuum theories to model the nematic liquid crystals:
the Oseen-Frank theory\cite{LL}, the Landau-de Gennes theory\cite{dGP}, and the Ericksen theory\cite{E}. Among these models, only
the Landau-de Gennes model can describe both uniaxial and biaxial states of nematic liquid crystals. In this
model, the local state of nematic liquid crystals is described by a tensor valued order parameter $Q$ which takes value in the set
\begin{equation}
\nonumber \Qa=\{Q\in \mathbb{R}^{3\times 3}, Q_{ij}=Q_{ji}, Q_{ii}=0\}.
\end{equation}
The nematic liquid crystal is called isotropic at point $\xx$ when $Q(\xx)=0$; when $Q(\xx)$ has two equal
non-zero eigenvalues, it is called uniaxial and when $Q(\xx)$ has three distinct eigenvalues, it is called biaxial.

Physically, the tensor $Q$ can be interpreted as the second-order traceless moment of the orientational distribution
function $f$, which is
\begin{equation*}
Q(\xx)=\int_{\mathbb{S}_2}(\mm\mm-\frac{1}{3}I)f(\xx,\mm)d\mm.
\end{equation*}
Under this interpretation, all the eigenvalues of $Q$ should belong to the interval $(-\f{1}{3},\f{2}{3})$. In other words,
$Q$ should satisfy the natural physical constraint that
\begin{equation}\label{constraint}
Q\in \Qp:=\left\{Q\in\Qa:\quad \lam_1(Q), \lam_2(Q), \lam_3(Q)\in \left(-\f{1}{3},\f{2}{3}\right)\right\}.
\end{equation}
\par

The Landau-de Gennes energy, in the simplest form, can be written as
\begin{equation}\label{FLG}
\mathcal{F}_{LG}[Q]=\int_{\Om}\Big\{\f{L}{2}|\nabla Q(\xx)|^2+\mathfrak{f}_{p}(Q(\xx))\Big\}d\xx,
\end{equation}
where $\mathfrak{f}_{p}$ is the bulk energy density that accounts for bulk effects which takes the following polynomial form:
\begin{equation}
\nonumber \mathfrak{f}_{p}(Q)=\f{\bar a}{2}tr(Q^2)-\f{\bar b}{3}tr(Q^3)+\f{\bar c}{4}tr(Q^2)^2.
\end{equation}
Here $\bar a,\bar b,\bar c $ are constants dependent on materials and temperature with $\bar b>0, \bar c>0$. It is well-known that, when $\bar a<0$, $\mathfrak{f}_{p}(Q)$ attains its minimum on the manifold
\beno
\cN_{LG}=\Big\{Q\in \BQ:Q=s^{+}(\nn\otimes \nn-\f{1}{3}\II),\, \nn\in \mathbb{R}^3,|\nn|=1\Big\},
\eeno
where $s^{+}=\f{\bar b+\sqrt{\bar{b}^2+4\bar{a}\bar{c}}}{4\bar c}$. It is easy to see that $\cN_{LG}$ is a smooth submanifold of $\BQ$,
homemorphic to the real projective plane $\mathbb{R}\BP^2$, and contained in the sphere $\Big\{Q\in \BQ:|Q|=\sqrt{\f{2}{3}}s^{+}\Big\}$.
Critical points of Landau-de Gennes functional satisfy the Euler-Lagrange equation
\begin{equation}\label{eq:EL-LdG}
 L\Delta Q=\bar aQ-\bar b(Q^2-\frac{1}{3}|Q|^2I)+\bar cQ|Q|^2.
\end{equation}

Landau-de Gennes energy (\ref{FLG}) are widely used to study the static configurations
and dynamic behaviors of liquid crystal material in various settings,  see \cite{BaP, C, GM, MZ} and reference therein for examples.
Special solutions to Euler-Lagrange equation (\ref{eq:EL-LdG}) and their stabilities are of particular interests on both physical and analytical aspects.
However, due to the strong nonlinearity, there are only a few nontrivial solutions can be expressed explicitly. One of these examples
is the radial symmetric solution in a ball or in $\mathbb{R}^3$, named
hedgehog solution. Hedgehog solution is regarded as a potential candidate profile for the isolated point defect in 3-D region. The property and stability
of this solution are well studied and it is shown that it not stable for large $a^2$ and stable for small $a^2$, we refer  \cite{RV, GM,	 Ma, La, INSZ0, INSZ1} and references therein for related results.



Another class of solutions of  the Euler-Lagrange equation (\ref{eq:EL-LdG}), which can be expressed explicitly, is
the set of radial symmetric solutions  in  $\mathbb{R}^2$, which can be regarded as possible profiles
for the isolated point defects of different degrees in 2-D region. Here ``radial" means that the order tensors along
the radial direction share the same eigenvectors. Precisely speaking,
the radial symmetric point defects in 2-D plane with degree-$k/2$ are solutions with form
\begin{align}\label{prsolu}
    Q(r,\varphi)
      &=u(r)F_1(\varphi)+v(r)F_2,
\end{align}
where $(r,\varphi)$ is the polar coordinate in $\mathbb{R}^2$, and
\begin{align}
  F_1=\left(
        \begin{array}{ccc}
          \cos k\varphi & \sin k\varphi & 0 \\
          \sin k\varphi & -\cos k\varphi & 0 \\
          0 & 0 & 0 \\
        \end{array}
      \right),\quad
  F_2=3\mathbf{e}_3\otimes\mathbf{e}_3-\II=\left(
        \begin{array}{ccc}
          -1 & 0 & 0 \\
          0 & -1 & 0 \\
          0 & 0 & 2 \\
        \end{array}
      \right).
\end{align}
The boundary condition of these solutions is taken to be
\begin{align}\label{boundary}
\lim_{r\to+\infty}Q(r,\varphi)=s_+(\nn(\varphi)\otimes \nn(\varphi)-\frac{1}{3}\II),\quad \nn(\varphi)=(\cos{\frac{k}{2}\varphi},\sin{\frac{k}{2}\varphi},0),~~
\end{align}
which has degree $\frac k 2$ about origin as an $\mathbb{R}\BP^2$-valued map. Here $k\in\mathbb{Z}\setminus\{0\}$.
 To satisfy the Euler-Lagrange equation (\ref{eq:EL-LdG}),
$u(r)$ and $v(r)$ need to satisfy the following ODE system \cite{DRSZ, HQZ} on $(0,R)$:
\begin{align} \label{ODE:LdG-u}
  u''+\frac{u'}{r}-\frac{k^2}{r^2}u&=u[-\bar{a}+2\bar{b}v+\bar{c}(6v^2+2u^2)] , \\ \label{ODE:LdG-v}
  v''+\frac{v'}{r}&=v[-\bar{a}-\bar{b}v+\bar{c}(6v^2+2u^2)]+\frac{\bar{b}}{3}u^2.
\end{align}
Existence of solution to the system (\ref{ODE:LdG-u})-(\ref{ODE:LdG-v}) with suitable boundary condition has been established in \cite{INSZ2}.
More importantly, it has also been proved
that the radial symmetric solutions are unstable for $k>1$ \cite{INSZ2} and stable for $k=1$ \cite{GWZZ, INSZ3}.

Although the phenomenological Landau-de Gennes theory is widely studied and has made great progresses, there is an apparent drawback in it that
the energy has no term to enforce the physical constraint $Q\in\Qp$. Therefore, the Landau-de Gennes energy may provide non-physical prediction which violate the
constraint (\ref{constraint}).
For this reason,
%
based on the homogeneous Maier-Saupe energy from molecular theory
\begin{align}\label{energy:Maier-Saupe}
\mathcal{A}[\rho]=\int_\BS \Big\{ \rho(\mm)\ln\rho(\mm) +\alpha\rho(\mm)\Big( \int_\BS |\mm\times\mm'|^2\rho(\mm')d\mm'\Big)\Big\} d\mm,
\end{align}
Ball-Majumdar\cite{BM} proposed a ``singular" bulk energy:
\begin{align}
\tilde{\mathfrak{f}}_{S}(Q)=\inf_{\rho \in \mathcal{P}_Q} \int_\BS \rho(\mm)\ln\rho(\mm) d\mm -\alpha |Q|^2,
\end{align}
 where
 \begin{align}
 \mathcal{P}_Q:=\Big\{\rho(\mm): \int_\BS\rho(\mm)d\mm=1, \int_\BS(\mm\otimes\mm-\frac13I)\rho(\mm)d\mm=Q\Big\}.
 \end{align}
An important feature of bulk energy $\tilde{\mathfrak{f}}_{S}$ is its logarithmic divergence when one of eigenvalues of $Q$ tends to $-1/3$ or $2/3$.
Thus, $\tilde{\mathfrak{f}}_{S}$ is only defined on $\Qp$. Readers are referred to \cite{BM,BP,Evans2016partial,GT} and references therein for results on variational problems with singular bulk potential modeling liquid crystal.

In \cite{HLW}, the authors obtained Ball-Majumdar's singular bulk energy $\tilde{\mathfrak{f}}_{S}$
 in another way, called Bingham closure,
which has been used to approximately calculate fourth-moment $\langle\mm\otimes\mm\otimes\mm\otimes\mm\rangle_{\rho}$
by using only the information of second moment $\langle\mm\otimes\mm\rangle_{\rho}$. For a given distributional function $\rho(\mm)$ in $\mathcal{A}_Q$,
the Bingham closure is to use the quasi-equilibrium distribution
\beq\nonumber
\rho_Q=\frac{1}{Z_Q}\exp(B_Q:\mm\mm),\qquad Z_Q=\int_{\mathbb{S}^2}\exp(B_Q:\mm\mm)d\mm
\eeq
to approximate $\rho$. Here, $B_Q\in\Qa$ is uniquely determined by $Q$ (for the proof of this fact, see \cite{BM} or \cite{LWZ})by the relation
\beq\label{rebq}
\frac{\int_{\mathbb{S}^2}(\mm\mm-\frac{1}{3}I)\exp(B_Q:\mm\mm)d\mm}{\int_{\mathbb{S}^2}\exp(B_Q:\mm\mm)d\mm}=Q,
\eeq
for $Q\in\Qp$.
Using $\rho_Q$ to replace $\rho$ in (\ref{energy:Maier-Saupe}), the Maier-Saupe energy (\ref{energy:Maier-Saupe}) reduces to
\begin{align}\label{energy:ms-0}
\mathfrak{f}_{S}(Q)=  Q:B_Q-\ln{Z_Q}-\frac{\alpha}{2}|Q|^2.
\end{align}
It is not difficult to check that $\tilde{\mathfrak{f}}_{S}$ and $\mathfrak{f}_{S}$ are indeed equivalent. Throughout this paper, we use the formulation in (\ref{energy:ms-0}).

The critical points of $\mathfrak{f}_{S}$ satisfy
\beq\label{crifb}
\f{\partial \mathfrak{f}_{S}}{\partial Q}:= B_Q-\alpha Q=0,
\eeq
of which solutions have been completely analyzed, see \cite{FS, LZZ, ZWFW} or Proposition \ref{prop:critical}
and \ref{prop:critical-stab}. In this paper, we consider the case of $\alpha>7.5$ in which
$Q=0$ is no longer a minimizer of  $\mathfrak{f}_{S}$ and $\mathfrak{f}_{S}$ attains its minimum only on the manifold
\beno
\cN=\Big\{Q\in \BQ:Q=s_2(\nn\otimes \nn-\f{1}{3}\II),\, \nn\in \mathbb{R}^3,|\nn|=1\Big\},
\eeno
with $s_2$ be a constant only depend on $\alpha$.

Combining the elastic energy, the total energy functional is given by(for simplicity we take $L=1$):
\beq\label{energy}
\mathcal{F}(Q,\nabla Q)=\int\left(\f{1}{2}|\nabla Q(\xx)|^2+\mathfrak{f}_{S}(Q)\right)d\xx,
\eeq
and the corresponding  Euler-Lagrange equation is
\beq\label{eleq}
\Delta Q=B_Q-\alpha Q.
\eeq

The aim of this paper is to study the profiles of point defects in the plane $\mathbb{R}^2$ as well as their
stabilities/instabilities using the $Q$-tensor framework with the singular energy (\ref{energy}).
For this, we consider the radial symmetric solutions of the Euler-Lagrange equation (\ref{eleq}) in a disk
$\mathbb{B}_R:=\{(x,y)| x^2+y^2\le R^2\}$ $(R\in(0,\infty)$ or $R=\infty$) on $\mathbb{R}^2$.
As introduced previously, $Q$ has the following form:
\begin{equation}\label{qform}
Q(r,\varphi)=u(r)F_1(\varphi)+v(r)F_2(\varphi),
\end{equation}
where $(r,\varphi)$ is the polar coordinate in $\mathbb{B}_R$, and
\begin{eqnarray}\label{def:F}
  F_1=\left(
        \begin{array}{ccc}
          \cos k\varphi & \sin k\varphi & 0 \\
          \sin k\varphi & -\cos k\varphi & 0 \\
          0 & 0 & 0 \\
        \end{array}
      \right),\qquad
  F_2=\left(
        \begin{array}{ccc}
          -1 & 0 & 0 \\
          0 & -1 & 0 \\
          0 & 0 & 2 \\
        \end{array}
      \right).
\end{eqnarray}
The boundary condition is taken as
\begin{equation}\label{BC:Q}
  Q(R,\varphi)=s_2(\ee^{(k)}_\varphi\otimes \ee^{(k)}_\varphi-\frac{1}{3}I_d),\quad
  \ee_\varphi=(\cos{\frac{k\varphi}{2}},\sin{\frac{k\varphi}{2}},0), \quad k=\pm1,\pm2,\ldots.
\end{equation}
Here if $R=\infty$, $p(R)$ is interpreted as $\lim_{r\to\infty} p(r)$ for a function $p$ whose limit at infinity exists.

Since $B_Q$ and $Q$ share the same eigenvectors(see Proposition \ref{prop:bingham}),  $B_Q$ can also be written as
\begin{equation}\label{bform}
  B_Q(r, \varphi)=f(r)F_1(\varphi)+g(r)F_2(\varphi).
\end{equation}
Here $(f,g)$ are only dependent on $(u,v)$ (and not depend on $\varphi$) through the following relations:
\begin{align*}
\frac{\int_{\mathbb{S}^2}(m_1^2-m_2^2)\exp({f(m_1^2-m_2^2)+g(2m_3^2-m_1^2-m_2^2)})
d\mm}{\int_{\BS}\exp({f(m_1^2-m_2^2)+g(2m_3^2-m_1^2-m_2^2)})d\mm}=2u,\\
\frac{\int_{\mathbb{S}^2}(2m_3^2-m_1^2-m_2^2)\exp({f(m_1^2-m_2^2)+g(2m_3^2-m_1^2
-m_2^2)})d\mm}{\int_{\BS}\exp({f(m_1^2-m_2^2)+g(2m_3^2-m_1^2-m_2^2)})d\mm}=6v.
\end{align*}
We write $f=f(u,v)$ and $g=g(u,v)$.

With the assumed formulation (\ref{qform}), $Q$ satisfies (\ref{eleq}) if and only if $(u,v)$ satisfies the following the ODE system on $(0,R)$:
\begin{align}\label{ODE:u}
  u''+\frac{u'}{r}-\frac{k^2}{r^2}u &=f(u,v)-\alpha u,  \\
 \label{ODE:v} v''+\frac{v'}{r} &=g(u,v)-\alpha v.
\end{align}
The boundary conditions (\ref{BC:Q}) at $x=R$ for $Q$ are equivalent to
\begin{equation}\label{BC:uv-R}
 u(R)=\frac{s_2}{2},\quad v(R)=-\frac{s_2}{6}.
\end{equation}
In addition, as $Q$ is smooth at $x=0$, we have
\begin{equation}\label{BC:uv-0}
u(0)=0,\quad v'(0)=0.
\end{equation}

Note that the ODE system (\ref{ODE:u})-(\ref{ODE:v}) is the Euler-Lagrange equation for the reduced energy functional:
\begin{align}\label{energyuv}
\mathcal{E}(u,v)=&\int_0^R \bigg\{\Big((u')^2+3(v')^2+\frac{k^2}{r^2}u^2\Big)-\ln{\int_{\mathbb{S}^2}
e^{f(u,v)(m_1^2-m_2^2)+g(u,v)(2m_3^2-m_1^2-m_2^2)}d\mm}\\
&\qquad+2f(u,v)u+6g(u,v)v-\alpha(u^2+3v^2)\bigg\}rdr.\nonumber
\end{align}
We can also define the reduced energy density corresponding to (\ref{energyuv}):
\begin{align}\label{densityuv}
e_1(u,v;r)=&(u')^2+3(v')^2+\frac{k^2}{r^2}u^2-\ln{\int_{\mathbb{S}^2}
e^{f(u,v)(m_1^2-m_2^2)+g(u,v)(2m_3^2-m_1^2-m_2^2)}d\mm}\\
\nonumber&+2f(u,v)u+6g(u,v)v-\alpha(u^2+3v^2).
\end{align}
\par
A solution $(u, v)$ of (\ref{ODE:u})-(\ref{ODE:v}) and (\ref{BC:uv-R})-(\ref{BC:uv-0}) is called a local minimizer of the 1-D reduced energy of (\ref{energyuv}) if
\begin{align}\label{def:min-1d}
\mathcal{J}(\mu,\nu) &\triangleq\int_0^R\frac{d^2}{dt^2}\Big|_{t=0}\Big( e_1(u+t\mu, v+t\nu; r)-e_1(u,v; r)\Big) rdr \\
\nonumber =2 &\int_0^{R}\Big\{({\partial_r\mu})^2 +\mu^2\Big(\frac{d}{de-c^2}-\alpha+\frac{k^2}{r^2}\Big)+3({\partial_r\nu})^2
+3\nu^2\Big(\frac{e}{de-c^2}-\alpha\Big) -\frac{2\sqrt{3}c\mu\nu}{de-c^2}\Big\} rdr\geq 0
\end{align}
for all $\mu, \nu\in C_c^\infty(0,R)$, where $c,d,e$ will be defined in the next section. One can also extend such definition to the space $H_0^1((0,R),rdr) \times H_0^1((0,R),rdr)\bigcap L^2((0,R),\frac{1}{r}dr)$ since $C_c^\infty((0,R),rdr)$ is dense in $H_0^1((0,R),rdr)$.

\par
For $V\in C_c^\infty(B_R)$, we  define
\begin{align}\label{ineq:main-0}
I(V) &\triangleq\frac{d^2}{dt^2}\int_{B_R}\Big\{ \frac{1}{2}|\nabla(Q+tV)|^2-\frac{\alpha}{2}|Q+tV|^2+B_{Q+tV}:(Q+tV)-\ln{Z_{Q+tV}}\\
\nonumber&\qquad\quad-[\frac{1}{2}|\nabla(Q)|^2-\frac{\alpha}{2}|Q|^2+B_{Q}:Q-\ln{Z_{Q}} ]\Big\}dx \\
\nonumber &=\int_{B_R}\Big\{ |\nabla V|^2-\alpha|V|^2 +\frac{\delta B}{\delta Q}(V):V \Big\} dx
\end{align}
where $\frac{\delta B}{\delta Q}(V)$ will be calculated in Section 4. The definition in the last line can be extended to all function $V\in H_0^1(B_R,\mathcal{Q})$. We say that a solution $Q$ to the Euler-Lagrange equation is a local minimizer of the
(\ref{energy}) if $\mathcal{I}_Q(V)\ge 0$ for all $V\in H_0^1(B_R,\mathcal{Q})$.

\subsection{Main results}
The main goal of this paper is to study the stabilities of degree-$k/2$ radial symmetric point defects in two dimensional plane.
The first step is to study existences of such kind solutions for all $k$. For this, we proved the existence of the solution which
satisfies $u>0,v<0$ of the above ODE systems in finite and infinite domains and proved the monotonicity of $(u,v)$ as
well as other qualitative properties. This is our first theorem.
\begin{theorem}\label{thm:existence}
For $\alpha>7.5$, $k\in \mathbb{Z}\backslash\{0\}$ and $R\in(0,\infty]$, there exists an solution
$u\in C^2([0,R])\cap C^\infty((0,R)),v\in C^\infty([0,R])$ of the ODE system (\ref{ODE:u})-(\ref{ODE:v})
with boundary condition (\ref{BC:uv-R})-(\ref{BC:uv-0}). In addition, the solution $(u,v)$ we obtained is a local minimizer of the reduced
functional $\mathcal{E}$ in (\ref{energyuv}) and satisfies
\beq
u>0,\quad v<0,\quad u'>0,\quad v'<0 \text{ on }(0,R).
\eeq
\end{theorem}
\begin{Remark}
Indeed, one can prove for all minimizers of  $\mathcal{E}(u,v)$ in (\ref{energyuv}), it hold that $u'>0,v'<0$ for all $r\in(0,R)$, see Proposition \ref{prop:monotonity}.
\end{Remark}


\par
For the question regarding to the stability of degree-$k/2$ radial symmetric point defects constructed above,
we show that they are stable for $|k|=1$ and are unstable for $|k|>1$. This is summarized in the following two theorems.
\begin{theorem}\label{thm:stability}
Let $(u,v)$ be the solution obtained in Theorem \ref{thm:existence}. When ${k=\pm 1}$, then the solution $Q=u(r)F_1+v(r)F_2$ is a local minimizer of the energy (\ref{energy}).
That is, for any perturbation $V\in H_0^1(\mathbb{B}_R, \mathbb{Q})$, it holds
\begin{align*}
    I(V) =\int_{B_R}\Big\{ |\nabla V|^2-\alpha|V|^2 +\frac{\delta B}{\delta Q}(V):V \Big\} dx\ge 0.
\end{align*}
Moreover, the equality holds if and only if
\begin{align}\label{kernel}
V=&\mu_0F_1+\nu_0 F_2
\end{align}
for some $(\mu_0,\nu_0)$ satisfying $\mathcal{J}(\mu_0,\nu_0)=0$.
\end{theorem}
\begin{theorem}\label{thm:instability}
When ${|k|>1}$ {and} ${R=\infty}$, the solution we construct in Theorem \ref{thm:existence} is unstable in the sense
that there exists $V\in H^1(\BR,\Qa)$ such that the second variation $I(V)<0$.
\end{theorem}

Theorem \ref{thm:stability} and Theorem \ref{thm:instability} indicate that 2-D radial symmetric solutions in singular energy share the same stability/instability properties
with the corresponding solutions in the classical Landau-de Gennes energy.  However, Remark 1.1 tells us that the eigenvalue $v$ is always decreasing in singular energy case.
This is different with the case of Landau-de Gennes energy, in which one have $v'<0$ for $\bar{b}^2>3\bar{a}\bar{c}$,
$v'\equiv0$ for $\bar{b}^2=3\bar{a}\bar{c}$ and $v'>0$ for $\bar{b}^2<3\bar{a}\bar{c}$.
This may give a new example which indicates that the Landau-de Gennes energy would give non-physical predictions on qualitative properties of order parameters for low temperature cases $\bar{b}^2<3\bar{a}\bar{c}$.

The main frameworks of proofs for Theorem \ref{thm:existence}-\ref{thm:instability} are similar
to those  in \cite{GWZZ, INSZ2, INSZ3} for the case of polynomial Landau-de Gennes energy. However, some new difficulties arise here. One of main difficulties is
that the relation between $(f,g)$ and $(u,v)$ are not apparent.
For this, we establish some key identities which are not only useful to prove the existences of solutions
but also important to study the monotonicity of $(u,v)$ and stabilities of profile solutions.
The proof of these identities involves the rotational gradient operator on the unit sphere and
suitably choosing of  vector fields to apply the integration by parts. On the other hand, to study the monotonicity of the solution,
some important inequalities need to be established.
We believe that these identities and inequalities may have independent interests and would be useful in studying other
related problems on $Q$-tensor models with Ball-Majumdar's singular energy.

\section{Preliminary analysis for the singular potential and the Bingham closure}
In this section, we present some results on the singular potential and Bingham closure. These results will play important roles in next sections.

\subsection{Existence and uniqueness of the Bingham map}
We recall some results on the Bingham closure. These results are already known in literatures.
\begin{proposition}\cite{BM, LWZ}\label{prop:bingham}
Bingham closure has the following properties:
\begin{enumerate}
  \item (Existence and uniqueness of $B_Q$). For any given $Q\in \Qp$, there exists a unique $B_Q\in\Qa$ such that the following relation holds:
  \beq\label{relation:B-Q}
\frac{\int_{\mathbb{S}^2}(\mm\mm-\frac{1}{3}I)\exp(B_Q:\mm\mm)d\mm}{\int_{\mathbb{S}^2}\exp(B_Q:\mm\mm)d\mm}=Q.
\eeq
Moreover, $B_Q$ also satisfies
      \beq\label{lag}
      B_Q:Q-w(B_Q)=\sup\limits_{B\in \Qa}(B:Q-w(B)),\qquad w(B):=\ln{\int_{\mathbb{S}^2}\exp(B:\mm\mm)d\mm}.
      \eeq
  \item Q and $B_Q$ are simultaneously diagonalizable, i.e., they share the same eigenvectors.
  \item For any $\delta>0$, there exists an positive constant $\Lambda=\Lambda(\delta)$ such that, when  all eigenvalues of $Q$ belong to $[-\frac{1}{3}+\delta,\f{2}{3}-\delta]$, then all the eigenvalues of $B_Q$ belong to $[-\Lambda(\delta),\Lambda(\delta)]$.
  \item Denote the maps between $\Qp$ and $\Qa$ as $B=B(Q): \Qp\rightarrow\Qa$ and $Q=Q(B): \Qa\rightarrow\Qp$. The Jacobian matrix $\nabla_B Q(B)$ is positive definite for any $B\in\Qa$. Consequently, $B(Q)$ is a smooth map from $\Qp$ to $\Qa$.
\end{enumerate}
\end{proposition}
From this proposition, we know that, for any $u\in(-1/2, 1/2), v\in(-1/6, 1/3)$ there exist a unique pair of $(f, g)$ such that for  $F_1, F_2$ defined in (\ref{def:F}),
\begin{align*}
\int_{\mathbb{S}^2}(\mm\mm-\frac{1}{3}I)\frac{\exp((fF_1+gF_2):\mm\mm)}{\int_{\mathbb{S}^2}\exp((fF_1+gF_2):\mm\mm)d\mm}d\mm=u F_1+v F_2,
\end{align*}
 or equivalently,
\begin{align}\label{relation:fg-u}
\frac{\int_{\mathbb{S}^2}(m_1^2-m_2^2)\exp({f(m_1^2-m_2^2)+g(2m_3^2-m_1^2-m_2^2)})
d\mm}{\int_{\BS}\exp({f(m_1^2-m_2^2)+g(2m_3^2-m_1^2-m_2^2)})d\mm}=2u,\\\label{relation:fg-v}
\frac{\int_{\mathbb{S}^2}(2m_3^2-m_1^2-m_2^2)\exp({f(m_1^2-m_2^2)+g(2m_3^2-m_1^2
-m_2^2)})d\mm}{\int_{\BS}\exp({f(m_1^2-m_2^2)+g(2m_3^2-m_1^2-m_2^2)})d\mm}=6v.
\end{align}
For given functions $f$ and $g$, we introduce the following notations:
\begin{align}\label{def:rho}
\rho_{f,g}(\mm)=\frac{\exp({f(m_1^2-m_2^2)+g(2m_3^2-m_1^2-m_2^2)})}{\int_{\mathbb{S}^2}\exp({f(m_1^2-m_2^2)+g(2m_3^2-m_1^2-m_2^2)})d\mm},
\end{align}
and define
\begin{align}
\langle h(\mm)\rangle_{f,g}=\int_{\BS}h(\mm)\rho_{f,g}(\mm)d\mm.
\end{align}
It will be simply written as $\langle h(\mm)\rangle$ when no confusion is caused.
Clearly, it holds that $\langle m_1^2-m_2^2\rangle=2u,\langle 2m_3^2-m_1^2-m_2^2\rangle=6v. $
It is not difficult to verify that if $u\to1/2(-1/2)$ for fixed $v$, then $f\to \infty(-\infty)$.

\subsection{Some identities and inequalities related to the Bingham closure}

Let $f, g, u$ and  $v$ be real numbers related by (\ref{relation:fg-u})-(\ref{relation:fg-v}) and $\rho=\rho_{f,g}$ be a probability distribution function on $\BS$ defined (\ref{def:rho}).
We introduce
\begin{align}\label{def:abh}
a=&2\langle m_1^2m_3^2\rangle;\qquad b=2\langle m_2^2m_3^2\rangle;\qquad h=2\langle m_1^2m_2^2\rangle;\\\label{def:c}
c=&\frac{\sqrt{3}}{2}\left(\left\langle\left(m_3^2-\frac{1}{3}\right)(m_1^2-m_2^2)\right\rangle
-\left\langle m_3^2-\frac{1}{3}\right\rangle\langle m_1^2-m_2^2\rangle\right);\\\label{def:d}
d=&\frac{3}{2}\left(\left\langle\left(m_3^2-\frac{1}{3}\right)^2\right\rangle-\left\langle m_3^2-\frac{1}{3}\right\rangle^2\right);\\\label{def:e}
e=&\frac{1}{2}\left(\left\langle(m_1^2-m_2^2)^2\right\rangle-\langle m_1^2-m_2^2\rangle^2\right).
\end{align}
The following lemmas give some basic algebraic relations, which will play crucial rules in
the proofs of Theorem \ref{thm:existence}-\ref{thm:instability}. These relations may also have independent
interests and we believe that they will be useful when studying some other problems on the singular bulk energy $\mathfrak{f}_S$.
\begin{lemma}
For the quantities defined above, the following relationship holds
\begin{align}\label{relation:fguv}
f&=\frac{a+b}{2ab}u+\f{3(a-b)}{2ab}v,\quad g=\f{a+b}{2ab}v+\f{a-b}{6ab}u,\\
\label{relation:uf}
u&=fh.
\end{align}
\end{lemma}
\begin{proof}
By Lemma 4.1 in \cite{HLW}, we have
\begin{align}
\frac{3}{2}Q=B_QQ+\frac13B_Q-\langle \mm\mm\mm\mm\rangle_{\rho_Q}:B_Q.
\end{align}
Taking $Q=diag\{u-v, -u-v, 2v\}$ and $B_Q=\{f-g, -f-g, 2g\}$ in the above identity with (\ref{relation:fg-u})-(\ref{relation:fg-v}) are satisfied,
and contracting with $\ee_3\otimes \ee_3$, we have
\begin{align}\label{Id:4-1}
3v=&4gv+\frac23g-\int_\BS m_3^2\left((2m_3^2-m_1^2-m_2^2)g+(m_1^2-m_2^2)f\right)\rho_{f,g} d\mm\nonumber\\
=&4gv+\frac23g-2g(2v+\frac13)+\frac32g(a+b)-\frac12f(a-b)\nonumber\\
=&\frac32g(a+b)-\frac12f(a-b).
\end{align}
On the other hand, by recalling the rotational gradient operator $\CR$ on unit sphere(see Page 1339 in \cite{WZZ1} or Appendix), we have for $B=\text{diag}\{f-g, -f-g, 2g\}$ that
\begin{align*}
\int_\BS (m_2^2-m_1^2) \exp(B:\mm\mm)d\mm=&\int_\BS \CR\cdot(m_2m_3, m_1m_3, 0) \exp(B:\mm\mm)d\mm\\
=&-\int_\BS (m_2m_3, m_1m_3, 0)\cdot\CR \exp(B:\mm\mm)d\mm\\
=&-2\int_\BS (m_2m_3, m_1m_3, 0)\cdot\Big(\mm\times (B\cdot \mm)\Big) \exp(B:\mm\mm)d\mm\\
=&-2\int_\BS \Big((f+3g)m_2^2m_3^2+(f-3g)m_1^2m_3^2\Big)\exp(B:\mm\mm)d\mm.
\end{align*}
Thus, we get
\begin{align}\label{Id:4-2}
2u=f(a+b)-3g(a-b);
\end{align}
Combining (\ref{Id:4-1}) and (\ref{Id:4-2}), we obtain (\ref{relation:fguv}).
\par
Similarly, we also have
\begin{align*}
2u=&\int_\BS (m_1^2-m_2^2)\exp(B:\mm\mm)d\mm\\
=&\int_\BS \CR\cdot(0, 0, m_1m_2) \exp(B:\mm\mm)d\mm\\
=&-\int_\BS (0, 0, m_1m_2)\cdot\CR \exp(B:\mm\mm)d\mm\\
=&-2\int_\BS (0, 0, m_1m_2)\cdot\Big(m\times (B\cdot m)\Big) \exp(B:\mm\mm)d\mm\\
=&4f\int_\BS m_1^2m_2^2\exp(B:\mm\mm)d\mm,
\end{align*}
which implies (\ref{relation:uf}).
\end{proof}
\begin{corol}\label{cor:ineq-f}
If $u>0$, then $f>0$.
\end{corol}

\begin{lemma}\label{lem:ineq-c}
It holds that:
\begin{itemize}
  \item $de-c^2>0$;
  \item If $f>0$, then $c<0$.
\end{itemize}
\end{lemma}
\begin{proof}
To prove the first inequality, we show that the equation $dt^2+2ct+e=0$ has no real root. Recalling the definition of $c,d$ and $e$, we have for any real number $t$ that
\begin{align*}
&dt^2+2ct+e\\
=&\left\langle\left(\sqrt{\f{3}{2}}\left(m_3^2-\frac{1}{3}\right)t+\sqrt{\f{1}{2}}\left(m_1^2-m_2^2\right)\right)^2\right\rangle
-\left(\left\langle\sqrt{\f{3}{2}}\left(m_3^2-\frac{1}{3}\right)\right\rangle t+\left\langle\sqrt{\f{1}{2}}\left(m_1^2-m_2^2\right)\right\rangle\right)^2\\
=&\left\langle \big(\varrho(t)-\left\langle \varrho(t)\big\rangle\right)^2\right\rangle \ge 0,
\end{align*}
where $\varrho(t)=\sqrt{\f{3}{2}}(m_3^2-\frac{1}{3})t+\sqrt{\f{1}{2}}(m_1^2-m_2^2)$. In addition,  the inequality holds strictly since $\rho(t)$ is not a constant.
\par
Next we show that $c<0$ if $f>0$.  This can be proved by applying Lemma 1 in \cite{ZWFW}. Here we present a proof for completeness.
Define
\begin{align}
  \mathfrak{a}(x)=\int_{0}^{2\pi}\exp(x\cos2\phi) d\phi.
\end{align}
We can find that $  \frac{ \mathfrak{a}'(x)}{ \mathfrak{a}(x)}$ is an increasing function, since
by Cauchy-Schwarz inequality,
\begin{align*}
\mathfrak{a}^2\Big( \frac{ \mathfrak{a}'(x)}{ \mathfrak{a}(x)}\Big)'=\int_0^{2\pi}\cos^22\phi \exp(r\cos2\phi)d\phi\int_0^{2\pi} \exp(r\cos2\phi)d\phi-\Big(\int_0^{2\pi}\cos2\phi \exp(r\cos2\phi)d\phi\Big)^2>0.
\end{align*}

Let $(m_1,m_2,m_3)=(\sin\theta\cos\phi, \sin\theta\sin\phi, \cos\theta)(\theta\in[0,\pi], \phi\in[0,2\pi))$. Then it holds
\begin{align*}
Z:=&\int_\BS \exp(f(m_1^2-m_2^2)+gm_3^2)d\mm=\int_{0}^\pi \int_0^{2\pi}\exp(g\cos^2\theta+f\sin^2\theta\cos2\phi)d\phi\sin\theta d\theta\\
=&\int_{0}^\pi \sin^2\theta \exp(g\cos^2\theta)\mathfrak{a}(f\sin^2\theta)\sin\theta d\theta.
\end{align*}
Similarly, we have
\begin{align*}
\langle m_3^2\rangle=&~Z^{-1}\int_{0}^\pi \int_0^{2\pi}\cos^2\theta \exp(g\cos^2\theta+f\sin^2\theta\cos2\phi)d\phi\sin\theta d\theta\\
=&~Z^{-1}\int_{0}^\pi \cos^2\theta \exp(g\cos^2\theta)\mathfrak{a}(f\sin^2\theta)\sin\theta d\theta,
\end{align*}
and
\begin{align*}
\langle m_1^2-m_2^2\rangle=&~Z^{-1}\int_{0}^\pi \int_0^{2\pi}\sin^2\theta\cos2\phi \exp(g\cos^2\theta+f\sin^2\theta\cos2\phi)d\phi\sin\theta d\theta\\
=&~Z^{-1}\int_{0}^\pi \sin^2\theta \exp(g\cos^2\theta)\mathfrak{a}'(f\sin^2\theta)\sin\theta d\theta,\\
\langle m_3^2( m_1^2-m_2^2)\rangle=&~Z^{-1}\int_{0}^\pi \int_0^{2\pi}\cos^2\theta\sin^2\theta\cos2\phi \exp(g\cos^2\theta+f\sin^2\theta\cos2\phi)d\phi\sin\theta d\theta\\
=&~Z^{-1}\int_{0}^\pi \cos^2\theta\sin^2\theta \exp(g\cos^2\theta)\mathfrak{a}'(f\sin^2\theta)\sin\theta d\theta.
\end{align*}
Thus, by the definition of $c$, we get
\begin{align*}
c=&~Z^{-2}\bigg[\Big( \int_{0}^\pi\cos^2\theta\sin^2\theta\exp(g\cos^2\theta)\mathfrak{a}'(f\sin^2\theta) \sin\theta
d\theta\Big)\Big(\int_{0}^\pi\exp(g\cos^2\theta)\mathfrak{a}(f\sin^2\theta)\sin\theta d\theta\Big)\\
&-\Big(\int_{0}^\pi\cos^2\theta\exp(g\cos^2\theta)\mathfrak{a}(f\sin^2\theta)\sin\theta d\theta\Big)
\Big(\int_{0}^\pi\sin^2\theta\exp(g\cos^2\theta)\mathfrak{a}'(f\sin^2\theta)\sin\theta d\theta\Big) \bigg]\\
=&~Z^{-2}\int_0^\pi\int_0^\pi\Big(\sin^2\theta\frac{\mathfrak{a}'(f\sin^2\theta)}{\mathfrak{a}(f\sin^2\theta)}-
\sin^2\tilde\theta\frac{\mathfrak{a}'(f\sin^2\tilde\theta)}{\mathfrak{a}(f\sin^2\tilde\theta)}\Big)\Big(\cos^2\theta-\cos^2\tilde\theta\Big)\\
&\qquad\cdot \mathfrak{a}(f\sin^2\theta)\mathfrak{a}(f\sin^2\tilde\theta)\exp(g\cos^2\theta+g\cos^2\tilde\theta)\sin\theta d\theta \sin\tilde\theta d\tilde\theta\\
<& ~0,
\end{align*}
where we have used the fact that $ \frac{ \mathfrak{a}'(x)}{ \mathfrak{a}(x)}$ is an increasing function. The proof is finished.
\end{proof}

\subsection{Critical points of the singular bulk energy}
We recall some results on the critical points of the singular bulk energy:
\begin{align}\label{energy:ms}
\mathfrak{f}_{S}(Q)=B_Q:Q-\ln Z_Q -\alpha |Q|^2.
\end{align}
Define a monotonic increasing function $s_2:(-\infty,+\infty)\mapsto(-0.5, 1)$ as
\begin{align*}
s_2(\eta)=\frac{\int_{-1}^1(3x^2-1)\exp(\eta x^2)dx  }{2\int_{-1}^1\exp(\eta x^2)dx}.
\end{align*}
Then all the critical points of (\ref{energy:ms}) are characterized by the following proposition.
\begin{proposition}\cite{FS, LZZ,ZWFW}\label{prop:critical}
All the critical points of (\ref{energy:ms}) are given by
\begin{align}
Q=s_2(\eta)(\nn\otimes\nn-\frac13\II),\quad \nn\in\BS,
\end{align}
where $\eta$ and $\alpha$ satisfies the following relation:
\beq\label{eta}
\eta=\alpha s_2(\eta).
\eeq
For all $\alpha>0$, $\eta=0$ is always a solution of (\ref{eta}). In addition, we have
\begin{enumerate}
  \item when $\alpha<\alpha^*$, $\eta = 0$ is the only solution of (\ref{eta});
  \item when $\alpha=\alpha^*$, besides $\eta = 0$ there is another solution $\eta=\eta^*$ of (\ref{eta});
  \item  when $\alpha>\alpha^*$, besides $\eta=0$ there are two solutions $\eta_1>\eta^*>\eta_2$ of (\ref{eta}).
\end{enumerate}
\end{proposition}
Furthermore, the stability/instability of critical points have also been clearly classified.
\begin{proposition}\cite{WZZ1, ZW}\label{prop:critical-stab}
\begin{enumerate}
  \item   When $\alpha<\alpha^*$, $Q=0$ is the only critical point. Thus, it is stable;
  \item When $\alpha^*\le \alpha<7.5$,  the solution corresponding to $\eta = 0$ and $\eta=\eta_1$ are both stable;
  \item  When $\alpha>7.5$, the solution corresponding to $\eta=\eta_1$ is the only stable solution.
\end{enumerate}
\end{proposition}
Throughout  this paper, we choose $\alpha>7.5$ and $s_2=s_2(\eta_1)$. In this setting, $Q=s_2(\nn\otimes\nn-\frac13\II)$ is the global minimizer of $\mathfrak{f}_{S}$ and therefore
\begin{align}\label{lowbound}
\mathfrak{f}_{S}(Q)\ge \mathfrak{f}_{S}\big(s_2(\nn\otimes\nn-\frac13\II)\big).
\end{align}

\section{Existence of radial symmetric solutions}\label{sec:exist}
This section is devoted to proving Theorem \ref{thm:existence}. We first prove the existence of solutions for the
finite domain case ($R<\infty$). Next, in Subsection \ref{subsec:mono},
we prove the monotonicity of the obtained solution. In Subsection \ref{subsec:exist-infinite}, we construct solutions
for the case $R=\infty$ based on the solutions obtained in the finite domain case.

\subsection{Existence of solution to the ODE system on finite domain}\label{subsec:exist-finite}

This functional is defined on admissible set
\begin{align}
\mathcal{T}=\left\{(u,v):~~\sqrt{r}u',\sqrt{r}v',\f{u}{\sqrt{r}},\sqrt{r}v \in L^2(0,R),u(R)=\frac{s_2}{2},v(R)=-\frac{s_2}{6}, uF_1+vF_2\in \Qp\right\}.
\end{align}

\begin{proposition}\label{prop:exist-finite}
For finite $R>0$, there exist a solution $(u,v)$ of (\ref{ODE:u})-(\ref{ODE:v}) with boundary condition (\ref{BC:uv-R})-(\ref{BC:uv-0}) satisfies:
\begin{itemize}
\item $u,v\in C^2([0,R])\cap C^\infty((0,R))$;
\item $u>0,\, v<0,\, 3v+u<0$ in $(0,R)$;
\item $(u,v)$  is a local minimizer of
the reduced energy $\mathcal{E}$ in (\ref{energyuv}). That is
\begin{align*}
\mathcal{J}(\mu,\nu):=&\frac{d^2}{dt^2}\mathcal{E}(u+t\mu,v+t\nu)\Big|_{t=0}\\
   \nonumber = &2 \int_0^{R}\Big\{({\partial_r\mu})^2 +\mu^2\Big(\frac{d}{de-c^2}-\alpha+\frac{k^2}{r^2}\Big)+3({\partial_r\nu})^2
+3\nu^2\Big(\frac{e}{de-c^2}-\alpha\Big) -\frac{2\sqrt{3}c\mu\nu}{de-c^2}\Big\} rdr\geq 0.
\end{align*}
\end{itemize}
\end{proposition}
\begin{proof}
The proof follows closely the proof of \cite[Proposition 3.1]{INSZ2}, with some necessary modifications for singular bulk potential.
It will be divided into several steps:

\medskip

\textbf{Step 1: } Existence of minimizer of $\mathcal{E}$ on $\mathcal{T}_{-}=\{(u,v)\in\mathcal{T},v\leq 0\}$.

From (\ref{lowbound}), it holds that
\beq
\mathcal{E}(u,v)\geq \int_0^R \mathfrak{f}_S(diag\{u-v,-u-v,2v\}) rdr \ge \frac{R^2}{2} \mathfrak{f}_S(diag\{\frac23 s_2,-\frac13 s_2, -\frac13 s_2\}).
\eeq
which indicates that $\mathcal{E} $ is bounded below
for all $(u,v)\in\mathcal{T}_{-}$. Therefore there exists a minimizing sequence $(u_n,v_n)$ such that
\beq
\lim\limits_{n\rightarrow +\infty}\mathcal{E}(u_n,v_n)\rightarrow \inf\limits_{(u,v)\in \mathcal{T}_{-}}\mathcal{E}(u,v).
\eeq
And we also have there exists $(u,v)$ such that a subsequence of $(u_n,v_n)$ (also denoted by $(u_n,v_n)$) converges weakly to $(u,v)$ in $[H^1((0,R),rdr)\cap L^2((0,R),\f{dr}{r})]\times H^1((0,R),rdr)$. Use the weak lower semi-continuity of the Dirichlet term in $\mathcal{E}$, we get
\beq
\mathcal{E}(u,v)\leq \liminf\limits_{n\rightarrow +\infty} \mathcal{E}(u_n,v_n)
\eeq
Therefore $(u,v)$ is a global minimizer of $\mathcal{E}$ on $\mathcal{T}_-$. Note that as $\lam(Q)\to -\f13\text{ or }\f23$ , the bulk energy density will blow up. Thus by a truncation argument, there is some positive $\e>0$ such that $u\in(-\f12+\e,\f12-\e)$ and $v\in (-\f16+\e,0]$. This guarantees that we can calculate the first variation of the energy functional and conclude that the minimizer $(u,v)$ satisfies the following equality and inequality distributionally on $(0,R)$:
\begin{eqnarray}
  \label{odeuv3u}u''+\frac{u'}{r}-\frac{k^2}{r^2}u &=f(u,v)-\alpha u,  \\
\label{odeuv3v}   v''+\frac{v'}{r} &\geq g(u,v)-\alpha v,
\end{eqnarray}
with boundary condition (\ref{BC:uv-R})-(\ref{BC:uv-0}).
\par
According to (\ref{relation:uf}), we can deduce that $f(-u, v)=-f(u,v)$. Thus, $\mathcal{E}(u,v)=\mathcal{E}(-u,v)$.
Then we can deduce that $(|u|,v)$ is also
a minimizer of the reduced energy $\mathcal{E}$ and therefore it also satisfy (\ref{ODE:v}). Applying strong maximum
principle to the equation \eqref{ODE:u}(here we use $f=\frac{u}{h}$), we obtain that $|u|>0$ on $(0,R)$,
thus we have
\begin{equation}
u(r)>0,  \quad\forall x \in (0,R).
\end{equation}
On the open set $\{v<0\}$, the inequality in (\ref{odeuv3v}) becomes an equality
and by standard regularity arguments and bootstrap method we conclude that
$(u,v)\in C^{\infty}((0,R]\cap \{v<0\})\times C^\infty((0,R]\cap \{v<0\})$.
Moreover, since $u,v$ are continuous in $(0,R]$, by \eqref{odeuv3u} we have $u\in C^2((0,R])$.

\medskip
\textbf{Step 2}: Prove $\limsup\limits_{r\rightarrow 0^+}v(r)<0$.


Consider the reduced energy $\mathcal{E}$ \eqref{energyuv}, take the derivative on $v$, we get
\begin{equation}
\f{\partial \mathcal{E}}{\partial v}=6g-6\alpha v=6(\frac{a+b}{2ab}-\alpha)v+\frac{a-b}{ab}u
\end{equation}
We will prove the argument by contradiction. Assume $\limsup\limits_{r\rightarrow 0^+}v(r)=0$.
Since $u>0$ on $(0,R)$, we have $f>0$ and therefore $a-b>0$. Moreover, when $u,v$ are
both near 0, $f,g$ are also near 0 since the map $B(Q):\mathcal{Q}_{phy}\rightarrow\mathcal{Q} $
is smooth. When $u=0,v=0$, by direct calculation we have that $a=b=\f{2}{15}$. Next we can prove
the following argument.\\
\par
\textbf{Claim:} when $\alpha>7.5$, there exists $\delta<0$ such that $\mathcal{E}$ is monotone increasing in $v$ when $v\in[\delta,0]$.\\
\textbf{Proof of the claim: }
Assume $\alpha>7.5$, then there exists $\delta_0<0,\epsilon_0>0$ such that when
$u\in[0,\epsilon_0],v\in[\delta_0,0]$, $\frac{a+b}{2ab}-\alpha<0$. In such case it
is clear that $\frac{\partial \mathcal{E}}{\partial v}>0$.
\par
Then we fix $\epsilon_0$. For any $u\in(\epsilon_0,\frac{1}{2}),v\in[\delta_0,0]$, there
exists $C_1>0,C_2>0$ such that $a-b>C_1$ and $\max\{a,b\}<C_2$. Therefore there
exists $C_3>0$ such that when $\frac{a+b}{2ab}-\alpha>0$,
\begin{equation*}
\frac{(a-b)u}{6(a+b-2\alpha ab)}>C_3.
\end{equation*}
Hence, there exists $\delta<0$ such that $\frac{\partial \mathcal{E}}{\partial v}>0$ when $v\in[\delta,0]$. Thus the claim is proved.\\
\par

Then we can take a interval $(R_1,R_2)\in (0,R)$ satisfy $v>\delta$ in $(R_1,R_2)$, $v(R_2)=\delta$
and either $R_1=0$ or $v(R_1)=\delta$ and set the value of $v$ to be $\delta$ on $(R_1,R_2)$ to get
a new $\tilde{v}$. It is clear that $\mathcal{E}(u,\tilde{v})<\mathcal{E}(u,v)$, which contradicts the
fact that $(u,v)$ is a global minimizer of $\mathcal{E}$ on $\mathcal{T}_-$.

\medskip
\textbf{Step 3}: Prove that $3v+u<0$ in $(0,R)$.

Let $w=\frac{v}{u}+\frac{1}{3}$, direct calculation combined with (\ref{ODE:v}) and (\ref{relation:fguv}) gives that
\beq\label{eqw}
w''+(\f{1}{r}+\f{2u'}{u})w'+\f{a-b}{6ab}(\f{3v}{u}-\f{1}{3})w\geq -\f{v}{u}(\f{k^2}{r^2}).
\eeq
We have $-\f{v}{u}(\f{k^2}{r^2})\geq 0$ and $\f{a-b}{6ab}(\f{3v}{u}-\f{1}{3})<0$. We also have
that $w(R)=0$ and $w(0)<0$ due to step 2. Then we can apply strong maximum principle for (\ref{eqw}) to get $w<0$ on $(0,R)$.
\par

\medskip
\textbf{Step 4}: Completing the proof of the proposition.

We already get the existence of solution that satisfies $(u,v)\in C^\infty(0,R)\cap C^2(0,R]$ and $u>0,v<0$.
We still need to verify the boundary condition and prove the regularity of $(u,v)$ up to boundary $r=0$. The proof we show here mimic the arguments in the proof of \cite[Proposition 2.3]{INSZ2}. Since the global minimizer
of the reduced energy functional on $\mathcal{T}_-$ is also a local minimizer of the energy functional
on $\mathcal{T}$, $Q=u(r)F_1+v(r)F_2$ belongs to $H_{loc}^1(B_R(0),\Qp)$ and $Q$
satisfises (\ref{eleq}) on $B_R(0)\backslash\{0\}$. We conclude that $Q$ satisfies (\ref{eleq}) on $B_R(0)$
because a point has zero Newtonian capacity in 2D. Standard elliptic regularity theory implies the smoothness
of $Q$ inside $B_R(0)$. Thus $v(r)=\f{1}{6}Q\cdot F_2 $ can be written as a smooth function $V(x,y)$ in $B_R(0)$.
Then $v(r)=V(r,0)$ for all $r\in (0,R)$, and we have that $v(r)$ extends to a smooth even function on $(-r,r)$.
Therefore we get $v'(0)=0$ and $v(r)\in C^\infty([0,R))$.
\par
Recalling the fact that $u(r)\in H^1((0,R),rdr)\bigcap L_2((0,R),\f{dr}{r}) $,  we have by Cauchy-Schwartz inequality
\beq
u(r_2)^2-u(r_1)^2=\int_{r_1}^{r_2}u(r)u'(r)dr\leq \int_{r_1}^{r_2} \f{u(r)^2}{r}dr\int_{r_1}^{r_2}u'(r)^2r dr,\qquad \forall 0<r_1<r_2<R.
\eeq
Thus $u(r)$ is continuous on $[0,R]$ and $u\in L_2((0,R), \f{dr}{r})$ implies $u(0)=0$.
\par
Since $u$ and $v$ are both continuous on $[0,R]$, $f$ and $g$ are also continuous on $[0,R]$.
In particular, $\f{f(u,v)}{u}-\alpha=\f{1}{h}-\alpha$ is continuous on $[0,R]$. Thus, $u$ satisfies that
\beq
u''+\frac{u'}{r}-\frac{k^2}{r^2}u =(\f{1}{h}-\alpha)u,\qquad\text{for } r\in(0,R).
\eeq
 Then we have (by \cite[Proposition 2.2]{INSZ0}) that the function $U(r)=\f{u(r)}{r^{|k|}}$
 is continuously differentiable up to $r = 0$ and satisfies $U'(0)=0$. Thus, it's straightforward to deduce that $\f{u'}{r}-\f{k^2u}{r^2}$ is continuous on $[0,R]$, and therefore
 $u\in C^2([0,R])$. The proof is complete.\end{proof}

\subsection{Monotonicity of the solution}\label{subsec:mono}
In this subsection, we prove the following proposition, which implies the solution obtained in Subsection \ref{subsec:exist-finite} is monotonic.
\begin{proposition}\label{prop:monotonity}
If $(u,v)$ is a local minimizer of the reduced energy $\mathcal{E}$, then we have $u'(r)>0, v'(r)<0$ for $r\in(0,R)$.
\end{proposition}
\begin{rmk}
In this proposition, we do not assume $(u,v)$ to be the solution constructed in Proposition \ref{prop:exist-finite}. The monotonicity property is satisfied by all local minimizers of $\mathcal{E}(u,v)$.
\end{rmk}
\begin{proof}
Firstly we recall from (\ref{relation:fg-u})-(\ref{relation:fg-v}) the following relations:
\begin{align*}
\frac{\int_{\mathbb{S}^2}(m_1^2-m_2^2)\exp({f(m_1^2-m_2^2)+g(2m_3^2-m_1^2-m_2^2)})
d\mm}{\int_{\BS}\exp({f(m_1^2-m_2^2)+g(2m_3^2-m_1^2-m_2^2)})d\mm}=2u,\\
\frac{\int_{\mathbb{S}^2}(2m_3^2-m_1^2-m_2^2)\exp({f(m_1^2-m_2^2)+g(2m_3^2-m_1^2
-m_2^2)})d\mm}{\int_{\BS}\exp({f(m_1^2-m_2^2)+g(2m_3^2-m_1^2-m_2^2)})d\mm}=6v.
\end{align*}
Considering two sides of the first equality as functions of independent variables $(u,v)$ and taking derivative with respect to $u$, we get:
\begin{align*}
&f_u\Big(\langle(m_1^2-m_2^2)^2\rangle-\langle m_1^2-m_2^2\rangle^2\Big)\\
&\qquad+g_u\Big(\langle(m_1^2-m_2^2)(2m_3^2-m_1^2-m_2^2)\rangle-\langle 2m_3^2-m_1^2-m_2^2\rangle\langle m_1^2-m_2^2\rangle\Big)
=2,
\end{align*}
or equivalently
\begin{align*}
f_ue+\sqrt{3}g_uc=1.
\end{align*}
Similarly, taking derivative with respect to $u$ for the second equality, we have
\begin{align*}
f_uc+\sqrt3g_ud=0.
\end{align*}
Combine the above two equalities, we obtain:
\begin{align}
f_u=\frac{d}{de-c^2},\quad g_u=\frac{-c}{\sqrt{3}(de-c^2)}.
\end{align}
Taking derivative with respect to $v$ for both equalities, we can also obtain
\begin{align*}
f_ve+\sqrt{3}g_vc=0,\quad f_vc+\sqrt3g_vd=\sqrt3,
\end{align*}
which gives
\begin{align}
f_v=\frac{-\sqrt3c}{de-c^2},\quad g_v=\frac{e}{de-c^2}.
\end{align}

Taking derivative with respect to $r$ on the ODE system (\ref{ODE:u})-(\ref{ODE:v}),
we have
\begin{align}\label{u3dev}
u'''+\frac{u''}{r}-\frac{(1+k^2)u'}{r^2}+\frac{2k^2u}{r^3}&=(f_u-\alpha)u'+f_vv',\\
\label{v3dev}v'''+\frac{v''}{r}-\frac{v'}{r^2}&=g_uu'+(g_v-\alpha)v'.
\end{align}
Therefore, we have
\begin{align}
\label{3de:u}u'''+\frac{u''}{r}-\frac{(1+k^2)u'}{r^2}+\frac{2k^2u}{r^3}&=(\frac{d}{de-c^2}-\alpha)u'-\frac{\sqrt3c}{de-c^2}v',\\
\label{3de:v}v'''+\frac{v''}{r}-\frac{v'}{r^2}&=-\frac{c}{\sqrt{3}(de-c^2)}u'+(\frac{e}{de-c^2}-\alpha)v'.
\end{align}

\medskip

Since $(u,v)$ is assumed to be a local minimizer of the reduced energy functional $\mathcal{E}$, the second variation should be nonnegative, i.e.,
\begin{align}\label{var:J}
\mathcal{J}(\mu,\nu)&=\frac{d^2}{dt^2}\mathcal{E}(u+t\mu,v+t\nu)\\
   \nonumber = &2 \int_0^{R}\Big\{({\partial_r\mu})^2 +\mu^2\Big(\frac{d}{de-c^2}-\alpha+\frac{k^2}{r^2}\Big)+3({\partial_r\nu})^2
+3\nu^2\Big(\frac{e}{de-c^2}-\alpha\Big) -\frac{2\sqrt{3}c\mu\nu}{de-c^2}\Big\} rdr\geq 0.
\end{align}
Let
\beq
\nonumber p(r)=uu',\qquad q(r)=\f{cv'}{u(de-c^2)}, \qquad \text{for }r\in[0,R].
\eeq
Firstly we show that $p(r)$ and $q(r)$ are non-negative. If $\{p(r)<0\}\cup \{q(r)<0\}\neq \emptyset$, we define
\beq
\nonumber \chi(r)=\mathbf{1}_{\{p(r)<0\}},\qquad \eta(r)=\mathbf{1}_{\{q(r)<0\}},
\eeq
and take $\mu=u'\chi$, $\nu=v'\eta$. As $u>0, c<0, de-c^2>0$ on $(0,R)$, we have that  $u'(r_1)=0$ when $p(r_1)=0$ and that  $v'(r_2)=0$ when $q(r_2)=0$.
Thus $\mu$ and $\nu$ are continuous on $[0,R]$. Substituting this into (\ref{var:J}),  we have
\begin{align}
\label{mono:ibp} 0 &\leq \int_0^{R}\Big\{({\partial_r\mu})^2 +\mu^2\Big(\frac{d}{de-c^2}-\alpha+\frac{k^2}{r^2}\Big)+3({\partial_r\nu})^2+3\nu^2\Big(\frac{e}{de-c^2}-\alpha\Big) -\frac{2\sqrt{3}c\mu\nu}{de-c^2}\Big\} rdr\\
\nonumber  &= \int_0^R\left\{  \chi\left((u'')^2 +(u')^2 (\f{d}{de-c^2}-\al+\f{k^2}{r^2})\right) +3\eta\left( (v'')^2+(v')^2(\f{e}{de-c^2}-\al)\right) -\f{2\sqrt{3}cu'v'\chi\eta}{de-c^2}  \right\}rdr\\
\nonumber &=\int_0^R \left\{  \chi (u'')^2+\chi u'(u'''+\frac{u''}{r}-\frac{u'}{r^2}+\f{2k^2u}{r^3}+\f{\sqrt{3}cv'}{de-c^2}) \right.\\
\nonumber &\qquad\qquad \left. +3\eta(v'')^2+3\eta v'(v'''+\f{v''}{r}-\f{v'}{r^2}+\f{cu'}{\sqrt{3}(de-c^2)})-\f{2\sqrt{3}cu'v'\chi\eta}{de-c^2} \right\}rdr
\end{align}
Here in the last line of \eqref{mono:ibp} we used equations \eqref{3de:u} and \eqref{3de:v}. Then we try to perform integration by parts for the above integral. Because $\chi$ and $\eta$ are characteristic functions of union of intervals, we need to deal with the boundary term carefully. Denote
\beqo
\{r\in(0,R):\,p(r)<0\}=\bigcup\limits_{i=1}^\infty(a_{i1},a_{i2})
\eeqo
Now we proceed in two different cases according to the sign of $u'(R)$.
\begin{case}
\normalfont Suppose $u'(R)\geq 0$, then it holds that
\beqo
a_{ij}<R,\quad u'(a_{ij})u''(a_{ij})a_{ij}=0 \quad \text{for all } i\in \mathbb{N},\, j=1,2.
\eeqo
So by fundamental theorem of calculus we have
\beq\label{ibp:u1}
\int_0^R \chi\left((u'')^2r+u'u'''r+u'u''\right)dr=\int_0^R \chi(u'u''r)'\,dr=0.
\eeq
\end{case}
\begin{case}
\normalfont Suppose $u'(R)<0$, then $R=a_{i2}$ for some $i$. By fundamental theorem of calculus we get
\beq\label{ibp:u2}
\int_0^R \chi\left((u'')^2r+u'u'''r+u'u''\right)dr=\int_0^R \chi(u'u''r)'\,dr=u'(R)u''(R)R.
\eeq
Note that the final expression contains a boundary term at $R$. However, the sign of this term is know. Recall that $(u(R),v(R))=(\f{s_2}{2},-\f{s_2}{6})$ is a local minimizer of the singular bulk energy \eqref{energy:ms}. Therefore if $u'(R)<0$ at $R$, then $u$ will be slightly larger than $\f{s_2}{2}$ near $R$, which leads to $f(u,v)-\al u\geq 0$ due to the stability property. By equation \eqref{ODE:u}, we get $u''(R)>0$ and therefore $u'(R)u''(R)R<0$.
\end{case}
Combining the two cases above we conclude that
\beq\label{ibp:u}
\int_0^R \chi\left((u'')^2r+u'u'''r+u'u''\right)dr=\int_0^R \chi(u'u''r)'\,dr\leq 0
\eeq
By similar argument as above we can obtain the following relation for $v$:
\beq\label{ibp:v}
\int_0^R \eta\left((v'')^2r+v'v'''r+v'v''\right)dr=\int_0^R \eta(v'v''r)'\,dr\leq 0
\eeq
Substituting \eqref{ibp:u} and \eqref{ibp:v} into \eqref{mono:ibp} leads to
\begin{align*}
0&\leq \int_0^R\left\{  \left(-\chi \f{(u')^2}{r}+\f{2k^2uu'\chi}{r^2}\right)-3\eta\f{(v')^2}{r} +\sqrt{3}\f{cu'v'r}{de-c^2}(\chi-\eta)^2\right\}dr\\
&=\int_0^R\left\{ -\left( 3\eta\f{(v')^2}{r^2}+\chi\f{(u')^2}{r^2}\right)  +\f{2k^2p\chi}{r^2}+\sqrt{3} (\chi-\eta)^2pqr\right\}dr
\end{align*}

Note that
\begin{align*}
& -\int_0^R\left\{3\eta\f{(v')^2}{r^2}+\chi\f{(u')^2}{r^2}\right\} rdr<0,\\
&\int_0^R \sqrt{3}(\chi-\eta)^2pq rdr=\left(\int_{\{p<0,q<0\}}+\int_{\{p>0,q<0\}\bigcup\{p<0,q>0\}}
+\int_{\{p>0,q>0\}}\right) \sqrt{3}(\chi-\eta)^2pqrdr\\
&\qquad\qquad\qquad\qquad\qquad\qquad~ =0+\text{non-positive terms} +0\leq 0,\\
&\int_0^R \f{2k^2p\chi}{r^2}dr=\int_{\{p<0\}} \f{2p}{r^2}dr\leq 0,
\end{align*}
where we obtain a contradiction. Thus, we have
\beq
\nonumber uu'\geq 0,\quad \f{cv'}{u(de-c^2)}\geq 0.
\eeq

Since $u(0)=0<u(R)$, we know  $u'(r)\ge 0$ and therefore $u(r)>0$  for all $r\in(0,R)$ by strong maximum principle. By Corollary \ref{cor:ineq-f} and Lemma \ref{lem:ineq-c},
we have $c(r)<0$, which implies $v'(r)\leq 0$ for $r\in(0,R)$.
\par
Then we prove $u'(r)>0$ for $r>0$. Assume there exists $r_0\in(0,R)$ such that $u'(r_0)=0$, then we have $u''(r_0)=0$ since $u'(r)\geq 0$ on $(0,R)$. By (\ref{3de:u}), we obtain
\beq
\nonumber u'''(r_0)=-\frac{\sqrt3c}{de-c^2}v'(r_0)-\f{2k^2u}{r_0^3}<0,
\eeq
which contradicts with $u'(r)\geq 0$ on $(0,R)$.
\par
Similarly, assume there exists $r_1$ such that $v'(r_1)=0$ and then we also have $v''(r_1)=0$. By (\ref{3de:v}), we have
\beq
\nonumber v'''(r_1)=-\frac{c}{\sqrt{3}(de-c^2)}u'(r_1)>0
\eeq
which yields a contradiction again. Thus, we have proved $u'(r)>0,v'(r)<0$ on $(0,R)$.
\end{proof}

\subsection{Existence of solution to the ODE system on infinite domain}\label{subsec:exist-infinite}
Now we give the proof for the existence of solution of (\ref{ODE:u})-(\ref{ODE:v}) and some properties for the infinite domain case: $R=\infty$.
\begin{proposition}
There exists a solution $(u,v)\in C^{\infty}(0,\infty)\times C^\infty(0,\infty)$  satisfies the ODE system
(\ref{ODE:u})-(\ref{ODE:v}) and the boundary condition (\ref{BC:uv-R})-(\ref{BC:uv-0}). And $(u,v)$ also satisfies
\begin{equation}
\nonumber 0<u<\f{s_2}{2},\quad -\f{s_2}{6}<v<0,\quad u'>0,\quad v'<0
\end{equation}
on $(0,\infty)$. Furthermore, $(u,v)$ is a local minimizer of the reduced energy $\mathcal{E}$ in (\ref{energyuv}).
\end{proposition}
\begin{proof}
The proof is similar in spirit to the proof of \cite[Proposition 4.1]{INSZ2}, with necessary modifications for singular bulk potential. So we will briefly sketch the proof here and emphasize on solving the difficulty caused by implicit relation between $(f,g)$ and $(u,v)$.
We first construct a solution $(u,v)$ on $(0,+\infty)$. Let $(u_n,v_n)$ be the solution we construct in
Proposition \ref{prop:exist-finite} for $R=n$, and the define the sequence $(\bar{u}_n,\bar{v}_n)$ as
\begin{eqnarray}
\nonumber \bar{u}_n(r)=\left\{
                \begin{array}{ll}
                  u_n(r), & \hbox{$r\in(0,n)$;} \\
                  \f{s_2}{2}, & \hbox{$r\geq n$.}
                \end{array}
             \right.\\
\nonumber \bar{v}_n(r)=\left\{
                \begin{array}{ll}
                  v_n(r), & \hbox{$r\in(0,n)$;} \\
                  -\f{s_2}{6}, & \hbox{$r\geq n$.}
                \end{array}
             \right.
\end{eqnarray}
Since $(u_n,v_n)$ are uniformly bounded, $(\bar{u}_n,\bar{v}_n)$ are also uniformly bounded,
and $(f(\bar{u}_n,\bar{v}_n),g(\bar{u}_n,\bar{v}_n))$ are uniformly bounded due to Proposition \ref{prop:bingham}(3).
For any compact interval $K\in (0,+\infty)$, we can show that there exists $n_0$ such that $(\bar{u}_n,\bar{v}_n)_{n\geq n_0}$
are uniformly bounded in $C^3(I)$ by standard regularity argument of the ODE system (\ref{ODE:u})-(\ref{ODE:v}). Then by Arzela-Ascolis theorem, we deduce
that $\bar{u}_n(\bar{v}_n)$ converges (up to a subsequence) in $C_{loc}^2(0,+\infty) $ to some
$u_\infty(v_\infty)\in C^2_{loc}(0,\infty)$. And $(u_\infty,v_\infty)$ satisfies the ODE system and have the following properties on $(0,R)$:
\begin{equation}
\nonumber 0\leq u\leq \f{s_2}{2},\quad -\f{s_2}{6}\leq v\leq 0, \quad u'\geq 0,\quad v'\leq 0.
\end{equation}
\par

The boundary condition at the origin can be verified easily by similar arguments in Step 4 of
the proof of Proposition \ref{prop:exist-finite}. And by the local minimizing property of $(u_n,v_n)$,
we deduce that $(u_\infty,v_\infty)$ is minimizing in the following sense: on any compact interval
$\omega$,  $\mathcal{E}(u_\infty,v_\infty)\leq \mathcal{E}(u_\infty+\delta_1,v_\infty+\delta_2)$
for all $(\delta_1(r),\delta_2(r))\in C_c^\infty(\omega)\times C_c^\infty(\omega)$ with $v_\infty
+\delta_2\leq 0$. Next we show the behavior of $(u_\infty,v_\infty)$ at $+\infty$.

\par
Since $u_\infty$ and $v_\infty$ are both non-decreasing functions, there exists $s_\infty\in[0,\f{s_2}{2}]$ and $t_\infty\in[-\f{s_2}{6},0]$ such that
\beq
\nonumber \lim\limits_{r\rightarrow +\infty} u_\infty(r)=s_\infty,\qquad \lim\limits_{r\rightarrow +\infty} v_\infty(r)=t_\infty.
\eeq
We claim that $(s_\infty,t_\infty)$ satisfy that
\begin{align*}
f(s_\infty,t_\infty)&=\alpha s_\infty,\\
g(s_\infty,t_\infty)&=\alpha t_\infty.
\end{align*}
Otherwise, suppose $f(s_\infty,t_\infty)> \alpha s_\infty$, by (\ref{ODE:u}), we have that
\beq\label{urinfty}
\f{(u'_\infty r)'}{r}=\f{k^2}{r^2}u_\infty+f(u_\infty,v_\infty)-\alpha u.
\eeq
Then there exists $R_0>0$ and $c>0$ such that for all $r>R_0$, we have
\beq\label{urinfty2}
\f{(u'_\infty(r) r)'}{r}>c.
\eeq
We integrate the above inequality from $R_0$ to $R_1$ to obtain that
\beq
u'_\infty(R_1)R_1-u'_\infty(R_0)R_0>\f{c}{2}(R_1^2-R_0^2).
\eeq
Therefore there exist $R_2>0$ and $c_1>0$ such that for all $r>R_2$, $u'_\infty(r)>c_1$, which contradicts with the boundedness of $u_\infty$. Similarly, if $f(s_\infty,t_\infty)< \alpha s_\infty$ or $g(s_\infty,t_\infty)\neq \alpha t_\infty$, we can get contradiction by the same arguments.

As $\alpha>7.5$,  (\ref{eta}) have only two non-negative solutions which are $\eta$ and $0$, we only have two possible values for $(s_\infty,t_\infty)$, and we need to show $(s_\infty,t_\infty)=(\f{s_2}{2},-\f{s_2}{6})$. We assume by contradiction that $u_\infty\equiv 0,v_\infty\equiv 0$. We consider the formula (\ref{iv0,01}), according to the locally minimizing property of $(u_\infty,v_\infty)$, we have
\beq
I_{0,01}(0,\mu_0^1)=\int_0^\infty\left( (\partial_r \mu_0^{(1)})^2+(\mu_0^{(1)})^2(-\alpha+\frac{d}{de-c^2}+\frac{k^2}{r^2})\right)rdr\geq 0.
\eeq
When $u= v = 0$, we have $c= 0$ and $e=\frac{2}{15}$. Thus $\f{d}{de-c^2}=\f{15}{2}<\alpha$. Thus, there exists $r_0>0$ and $\epsilon>0$ such that
\beq
-\alpha+\f{d}{de-c^2}+\f{k^2}{r^2}<-\epsilon.
\eeq
Take $\psi\in C_c^\infty(0,1)$ and let $\psi_n(r)=\psi(\f{r}{n})$, then we have
\beq
\int_0^n (\partial_r \psi_n)^2-\epsilon \psi_n^2rdr=\int_0^1 (\f{\partial \psi}{\partial r})^2-\epsilon n^2\psi^2rdr.
\eeq
Hence there exists an integer $N$ such that for all $n>N$, the integral above is negative. And we can take
\beq
\nonumber \mu_0^{(1)}=\left\{
                        \begin{array}{ll}
                          \psi_{2N}(r-r_0), & r\in[r_0,r_0+2N] \\
                          0, & \hbox{else.}
                        \end{array}
                      \right.
\eeq
Then we have $I_{0,01}(0,\mu_0^{(1)})$ is negative which yields a contradiction. Therefore we have proved $u_\infty(+\infty)=\f{s_2}{2}$, $v_\infty(+\infty)=-\f{s_2}{6}$.
\par
Finally, the smoothness of $(u,v)$ and the strict monotonicity ($u_\infty'>0$ and $v_\infty'<0$) just follows from the same arguments in the finite domain case.
\end{proof}

\section{Reduction of the variational inequality}\label{sec:reduce}
Section \ref{sec:reduce}-\ref{sec:instability} are devoted to studying the stability and instability
of radial solution constructed in Section \ref{sec:exist}.

We firstly calculate the second variation of the free energy $\mathcal{F}(Q,\nabla Q)$ at the critical point $Q=u(r)F_1+v(r)F_2$.
According to Lemma 3.1 in \cite{GWZZ}, we know that  $C_c^\infty(\mathbb{B}_R\backslash\{0\})$is dense in $H^1(\mathbb{B}_R)$,
thus we assume the test function $V\in C_c^\infty(\mathbb{B}_R\backslash \{0\},\Qa) $ for $R\in(0,\infty]$.
\begin{align}
  I(V) =& \frac{d^2}{dt^2}\Big|_{t=0}\int_{\BR}\Big\{ \frac{1}{2}|\nabla(Q+tV)|^2-\frac{\alpha}{2}|Q+tV|^2+B_{Q+tV}:(Q+tV)-\ln{Z_{Q+tV}} \Big\}dx \\
 \nonumber =&  \int_{\BR} \Big\{|\nabla V|^2-\alpha|V|^2 +2\frac{\delta B}{\delta Q}(V):V +\frac{\delta^2 B}{\delta Q^2}(V, V):Q\Big\}dx-\frac{d^2}{dt^2}\Big|_{t=0}\int_{R^2}\ln{Z_{Q+tV}}dx\\
  \nonumber  =& \int_{\BR}\Big\{ |\nabla V|^2-\alpha|V|^2 +\frac{\delta B}{\delta Q}(V):V \Big\} dx.
\end{align}
It is not an easy work to calculate $\frac{\delta B}{\delta Q}(V)$ , as $B$ is an implicit function of $Q$ by the relation (\ref{relation:B-Q}).
To avoid this difficulty, we  let $B=B+tU$, and expand Q as $$Q=Q+tV+t^2V_2+\cdots.$$
 According to the existing result in \cite{HLW}, we have that
\begin{align}
V=&\langle\mm\mm\mm\mm\rangle_{\rho_Q}:U-(\frac{1}{3}I+Q)(Q:U),\\
V_2=&(Q:U)V+\frac{1}{2Z_Q}\int(\mm\mm-\frac{1}{3}I)(\mm\mm:U)^2\exp(B:\mm\mm)d\mm\\\nonumber
&-\frac{Q}{2Z_Q}\int(\mm\mm:U)^2\exp(B:\mm\mm)d\mm.
\end{align}
Then we calculate
\begin{align*}
  I(V) =& \frac{d^2}{dt^2}\Big|_{t=0}\int_{\BR} \Big\{\frac{1}{2}|\nabla(Q+tV+t^2V_2+...)|^2-\frac{\alpha}{2}|Q+tV+t^2V_2+...|^2\\
  \nonumber &+(B+tU):(Q+tV+t^2V_2+...)-\ln{\int \exp((B+tU):\mm\mm)d\mm} \Big\}dx \\
   \nonumber =&  \int_{\BR}\Big\{ |\nabla V|^2+2\nabla Q\cdot \nabla V_2-\alpha|V|^2-2\alpha Q:V_2 +2B:V_2+2V:U\\
  \nonumber &-\frac{1}{Z_Q}\int \exp(B:\mm\mm)(U:\mm\mm)^2d\mm+\Big(\frac{1}{Z_Q}\int \exp(B:\mm\mm)(U:\mm\mm)d\mm\Big)^2\Big\}dx.
\end{align*}
Since $B-\alpha Q=\Delta Q$, we have
\begin{equation*}
\int_{\BR}\Big( 2\nabla Q\cdot \nabla V_2-2\alpha Q:V_2 +2B:V_2\Big) dx=\int_{\BR}\Big( 2\nabla Q\cdot \nabla V_2+2\Delta Q:V_2\Big) dx=0.
\end{equation*}
And we also observe that
\begin{equation*}
\frac{1}{Z_Q}\int \exp(B:\mm\mm)(U:\mm\mm)^2d\mm-\Big(\frac{1}{Z_Q}\int \exp(B:\mm\mm)(U:\mm\mm)d\mm\Big)^2=V:U.
\end{equation*}
Thus, $I(V)$ can be written as
\begin{align*}
  I(V) =&  \int_{\BR}\Big\{ |\nabla V|^2-\alpha|V|^2+V:U\Big\}dx,
\end{align*}
where $U$ is determined by
\begin{align}
V=\CL_Q(U):=&\langle\mm\mm\mm\mm\rangle_{\rho_Q}:U-(\frac{1}{3}I+Q)(Q:U).
\end{align}
The following lemma implies the inverse of $\CL_Q$ on $\Qa$ is well-defined, and then we can write
\begin{align}\label{energy:U}
  I(V) =&  \int_{\BR}\Big\{ |\nabla V|^2-\alpha|V|^2+V:\CL_Q^{-1}(V)\Big\}dx.
\end{align}
\begin{lem}
For any $Q\in\Qp$, $\CL_Q$ is a self-adjoint, positive operator from $\Qa$ to $\Qa$. 
\end{lem}
\begin{proof}For $V_1, V_2\in \Qa$, we have
\begin{align*}
\CL_Q(V_1):V_2&= V_2:\Big((Q+\frac{1}{3}\II)\cdot{V_1}-V_1:M_Q^{(4)}\Big)\\
&=\int_\BS(V_1:\mm\mm)(V_2:\mm\mm)\rho_Q(\mm)\ud\mm-(V_1:Q)(V_2:Q).\nonumber
\end{align*}
In addition, by Cauchy-Schwartz inequality, it holds for $V_1\neq0$ that
\begin{align*}
\CL_Q(V_1):V_1&=\int_\BS(V_1:\mm\mm)^2\rho_Q(\mm)\ud\mm-(V_1:Q)^2\\
&=\frac12\int_\BS(V_1:\mm\mm-V_1:\mm'\mm')^2\rho_Q(\mm)\rho_Q(\mm')\ud\mm\ud\mm'>0.
\end{align*}
This gives rise to our lemma.
\end{proof}

As in \cite{GWZZ}, at any point $(r,\varphi)$, we write $V$ as a linear combination of the normal orthogonal basis $\{E_i\}$:
\begin{align}
V=\sum\limits_{i=0}^4 w_i(r,\varphi) E_i,
\end{align}
where $w_i\in C_c^\infty(\mathbb{B}_R\backslash \{0\})$, and $\{E_i\}$ is defined by
\begin{eqnarray}\label{def:ei}
  E_0&=&\sqrt{\frac{3}{2}}\left(
                                \begin{array}{ccc}
                                  -\frac{1}{3} & 0 & 0 \\
                                  0 & -\frac{1}{3} & 0 \\
                                  0 & 0 & \frac{2}{3} \\
                                \end{array}
                              \right)=\frac{1}{\sqrt{6}}F_2,
\\
  \nonumber E_1&=&\frac{1}{\sqrt{2}}\left(
        \begin{array}{ccc}
          \cos \varphi & \sin \varphi & 0 \\
          \sin \varphi & -\cos \varphi & 0 \\
          0 & 0 & 0 \\
        \end{array}
      \right)=\frac{1}{\sqrt{2}}F_1,\\
\nonumber E_2&=&
        \frac{1}{\sqrt{2}}\left(
        \begin{array}{ccc}
          -\sin \varphi & \cos \varphi & 0 \\
          \cos \varphi & \sin \varphi & 0 \\
          0 & 0 & 0 \\
        \end{array}
      \right),\\
\nonumber E_3&=&
\frac{1}{\sqrt{2}}\left(
  \begin{array}{ccc}
    0 & 0 & 1 \\
    0 & 0 & 0 \\
    1 & 0 & 0 \\
  \end{array}
\right),\qquad E_4 =
\frac{1}{\sqrt{2}}\left(
  \begin{array}{ccc}
    0 & 0 & 0 \\
    0 & 0 & 1 \\
    0 & 1 & 0 \\
  \end{array}
\right).
\end{eqnarray}

It is easy to check that for $Q=u(r)F_1+v(r)F_2$ it holds
\begin{align}
\CL_Q(E_i):E_j=0,\qquad \text{for }i=0,1,2;~~ j=3,4.
\end{align}
Thus, we also have $\CL_Q^{-1}(E_i):E_j=0$  for $i=0,1,2$ and $j=3,4$.
Therefore, $I(V)$ can be completely divided into two parts:
 $$I(V)= I( w_0 E_0+w_1E_1+w_2E_2)+I( w_3 E_3+w_4E_4)\triangleq I^A+I^B.$$
Then, to prove the stability is equivalent to show the non-negativity of both  $I^A$ and $I^B$.

\medskip

Now we will deduce some integral identities which would play an important role in the following proof.
In the sequel, we assume $\eta\in C_c^\infty(0,R)$ with $R\in (0,+\infty]$. According to (\ref{ODE:u})-(\ref{ODE:v}), we have
\begin{align}\nonumber
\mathcal{A}(\eta):=&\int_{0}^R \Big\{ (v\eta)_{r}^2+(\frac{a+b}{2ab}-\alpha)(v\eta)^2\Big\}rdr\\\nonumber
=&\int_{0}^R \Big\{ (v\eta)_{r}^2+(v''+\frac{v'}{r}-\frac{(a-b)u}{6ab})v\eta^2\Big\}rdr\\\nonumber
\nonumber =&(vv'\eta^2r)|_0^{R}+\int_0^{R} (v^2\eta_r^2-\frac{(a-b)uv}{6ab}\eta^2)rdr\\
=&\int_{0}^R \Big\{ (v\eta_{r})^2-\frac{(a-b)uv}{6ab}\eta^2\Big\}rdr,\label{eq:A}
\end{align}
and
\begin{align}
\nonumber \mathcal{B}(\eta):=&\int_{0}^R \Big\{ (u\eta)_{r}^2+(\frac{a+b}{2ab}-\alpha+\frac{3(a-b)v}{2abu}+\frac{k^2}{r^2})(u\eta)^2\Big\}rdr\\
  \nonumber =&\int_0^{R} [(\eta'u+u'\eta)^2+u\eta^2(u''+\frac{u'}{r})]rdr\\
  \nonumber =&(uu'\eta^2r)|_0^{R} +\int_0^{R} u^2\eta_r^2rdr\\
 =&\int_{0}^R (u\eta_{r})^2rdr.\label{eq:B}
\end{align}
On the other hand, from the equations (\ref{3de:u})-(\ref{3de:v}), we have
\begin{align}
\mathcal{C}(\eta)&:=\int_0^{R}\Big\{(v'\eta)_r^2 +(v'\eta)^2(\frac{e}{de-c^2}-\alpha)\Big\} rdr\nonumber\\
&=\int_0^{R}\Big\{(v''\eta+v'\eta')^2 +
v'\eta^2(v'''+\frac{v''}{r}-\frac{v'}{r^2}+\frac{c}{\sqrt{3}(de-c^2)}u')\Big\} rdr\nonumber\\
\nonumber
&=\int_0^{R}\Big\{(v'\eta')^2 -\frac{(v'\eta)^2}{r^2}
+\frac{c u'v'}{\sqrt{3}(de-c^2)}\eta^2\Big\} rdr+(rv''v'\eta^2)\big|_{0}^{R}\\
&=\int_0^{R}\Big\{(v'\eta')^2 -\frac{(v'\eta)^2}{r^2}
+\frac{c u'v'}{\sqrt{3}(de-c^2)}\eta^2\Big\} rdr,\label{eq:C}
\end{align}
and
\begin{align}
\mathcal{D}(\eta)&:=\int_0^{R}\Big\{(u'\eta)_r^2 +
(u'\eta)^2(\frac{d}{de-c^2}-\alpha+\frac{k^2}{r^2})\Big\} rdr\nonumber\\
&=\int_0^{R}\Big\{(u''\eta+u'\eta')^2 +
u'\eta^2\Big(u'''+\frac{u''}{r}-\frac{u'}{r^2}+\frac{2u}{r^3}+\frac{\sqrt3c}{de-c^2}v'\Big)\Big\} rdr\nonumber\\\nonumber
&=\int_0^{R}\Big\{(u'\eta')^2 -
\frac{(u'\eta)^2}{r^2}+{\frac{2uu'\eta^2}{r^3}}+\frac{\sqrt3cu'v'}{de-c^2}\eta^2\Big\} rdr+(ru''u'\eta^2)\big|_{0}^{R}\\
&=\int_0^{R}\Big\{(u'\eta')^2 -
\frac{(u'\eta)^2}{r^2}+{\frac{2uu'\eta^2}{r^3}}+\frac{\sqrt3cu'v'}{de-c^2}\eta^2\Big\} rdr.\label{eq:D}
\end{align}
Moreover, we also have that
\begin{align}\nonumber
\mathcal{G}(\eta):=&\int_{0}^R \Big\{\Big(\frac{u\eta}{r}\Big)_{r}^2
+\Big(\frac{k^2}{r^2}-\alpha+\frac{1}{h}\Big)\frac{(u\eta)^2}{r^2}\Big\}rdr\\
  \nonumber =&\int_0^{R}\Big\{\Big(\frac{u\eta'}{r}+\big(\frac{u}{r}\big)'\eta\Big)^2
  +\frac{u\eta^2}{r^2}\Big(u''+\frac{u'}{r}\Big)\Big\}rdr\\\nonumber
=&\int_{0}^R\Big\{\Big(\frac{u}{r}\eta_{r}\Big)^2
+\frac{\eta^2}{r^4}(2ruu'-u^2)\Big\}rdr+\Big(\frac{u}{r}\Big)'u\eta^2\big|_0^R\\
=&\int_{0}^R\Big\{
\Big(\frac{u}{r}\eta_{r}\Big)^2+\frac{\eta^2}{r^4}(2ruu'-u^2)\Big\}rdr.\label{eq:E}
\end{align}

\section{Stability of radial symmetric solutions with degree $|k|=1$}\label{sec:stability}
In this section, we will prove Theorem \ref{thm:stability}. By the analysis in the previous section, it suffices to prove $I^A\ge0$ and $I^B\ge 0$ separatively.
This will be accomplished in the following two subsections.
\subsection{Non-negativity of $I^B$}

For any $\varphi\in[0,2\pi)$, let $\te_1=\cos\frac{\varphi}{2}\ee_1+\sin\frac{\varphi}{2}\ee_2, \te_2=\sin\frac{\varphi}{2}\ee_1-\cos\frac{\varphi}{2}\ee_2$, and
\begin{align*}
 &\tilde E_3=\cos\frac{\varphi}{2}E_3+\sin\frac{\varphi}{2}E_4=\frac1{\sqrt{2}}(\te_1\otimes \ee_3+\ee_3\otimes\te_1),\\
 & \tilde E_4=\sin\frac{\varphi}{2}E_3-\cos\frac{\varphi}{2}E_4=\frac1{\sqrt{2}}(\te_2\otimes \ee_3+\ee_3\otimes\te_2).
\end{align*}
\begin{lemma}
For $Q(r,\varphi)=u(r)F_1(\varphi)+v(r)F_2$, it hold that
\begin{align}\label{rela34}
 \CL_Q\tilde E_3:\tilde E_3=a>0,\quad \CL_Q\tilde E_3:\tilde E_4=0,\quad    \CL_Q\tilde E_4:\tilde E_4=b>0.
\end{align}
Here $a,b$ are the quantities defined in (\ref{def:abh}).
\end{lemma}
\begin{proof}Let $\tilde\mm=(\tilde m_1, \tilde m_2, \tilde m_3):=(\mm\cdot\te_1, \mm\cdot\te_2, \mm\cdot\ee_3)$. Then it is direct to verify that
\begin{align*}
  \CL_Q\tilde E_3:\tilde E_3=\frac{2\int_\BS \tilde m_1^2\tilde m_3^2
 \exp({f(\tilde m_1^2-\tilde m_2^2)+g(2\tilde m_3^2-\tilde m_1^2-\tilde m_2^2)})d\tilde\mm}{\int_{\mathbb{S}^2}
 \exp({f(\tilde m_1^2-\tilde m_2^2)+g(2\tilde m_3^2-\tilde m_1^2-\tilde m_2^2)})d\mm}=a>0.
\end{align*}
The other two equalities can be obtained similarly.
\end{proof}

From  (\ref{rela34}), we have the following identity:
\begin{equation*}
  \left(
    \begin{array}{ccc}
      \frac{\sin{\varphi}}{2} & -\cos{\varphi} & -\frac{\sin{\varphi}}{2} \\
      \cos^2{\frac{\varphi}{2}} & \sin{\varphi} & \sin^2{\frac{\varphi}{2}} \\
      \sin^2{\frac{\varphi}{2}} & -\sin{\varphi} & \cos^2{\frac{\varphi}{2}} \\
    \end{array}
    \right)
    \left(
      \begin{array}{c}
        \CL_Q E_3: E_3 \\
        \CL_Q E_3: E_4 \\
        \CL_Q E_4: E_4 \\
      \end{array}
    \right)
    =
    \left(
      \begin{array}{c}
        0 \\
        a \\
        b \\
      \end{array}
    \right)
\end{equation*}
Direct computation gives that
\begin{equation*}
\CL_Q
    \left(
      \begin{array}{c}
        E_3 \\
        E_4 \\
      \end{array}
    \right)=  \frac12\left(
    \begin{array}{ccc}
      a+b+(a-b)\cos\varphi & (a-b)\sin \varphi  \\
      (a-b)\sin \varphi & a+b-(a-b)\cos\varphi \\
    \end{array}
    \right)
    \left(
      \begin{array}{c}
         E_3 \\
        E_4 \\
      \end{array}
    \right),\nonumber
\end{equation*}
and hence
\begin{equation*}
\CL_Q^{-1}
    \left(
      \begin{array}{c}
        E_3 \\
        E_4 \\
      \end{array}
    \right)=  \frac1{2ab}\left(
    \begin{array}{ccc}
      a+b-(a-b)\cos\varphi & -(a-b)\sin \varphi  \\
      -(a-b)\sin \varphi & a+b+(a-b)\cos\varphi \\
    \end{array}
    \right)
    \left(
      \begin{array}{c}
         E_3 \\
        E_4 \\
      \end{array}
    \right).
\end{equation*}
 Therefore, by the definition of $I^B$, we get
\begin{eqnarray}
I^B&=&\int_0^{2\pi}\int_0^R\Big\{  \sum\limits_{i=3,4}\Big(({\partial_r w_i})^2 +\frac{1}{r^2}(\partial_\varphi w_i)^2\Big)-\alpha(w_3^2+w_4^2) \\ &&\qquad+\frac{a+b}{2ab}(w_3^2+w_4^2)-\frac{a-b}{2ab}\big((w_3^2-w_4^2)\cos\varphi+2w_3w_4\sin\varphi\big)\Big\}rdrd\varphi.\nonumber
\end{eqnarray}

As in \cite{PWZZ,GWZZ}, we introduce a complex function $z=w_3+iw_4$. Then we have
\begin{align}\label{Ib}
I^B
=&\int_0^{2\pi}\int_0^R\Big\{  |\partial_r z|^2 +\frac{1}{r^2}|\partial_\varphi z|^2-\alpha|z|^2 +\frac{a+b}{2ab}|z|^2-\frac{a-b}{2ab}Re(e^{-i\varphi}z^2)\Big\}rdrd\varphi.
\end{align}
Assume that
\begin{equation}
z(r,\varphi)=\sum\limits_{m=-\infty}^{+\infty}z_m(r)e^{im\varphi},\nonumber
\end{equation}
then
\begin{align}
  \nonumber Re(e^{-i\varphi}z^2) =& Re\big(e^{-i\varphi}\sum\limits_{l,m=-\infty}^{+\infty}e^{(m+l)\varphi}z_mz_l\big)
   = Re\big(\sum\limits_{l,m=-\infty}^{+\infty}e^{i(m+l-1)\varphi}z_mz_l\big).
\end{align}
 Substituting it into (\ref{Ib}), we get
\begin{align*}
I^B(z)=2\pi\int_0^{R}&\Big\{\sum\limits_{m=-\infty}^{+\infty}\Big[|\pa_rz_m|^2
+\f{m^2}{r^2}|z_m|^2+(\frac{a+b}{2ab}-\alpha)|z_m|^2\Big]-\frac{a-b}{2ab}\sum\limits_{m+l=1}Re(z_mz_l)\Big\}rdr.
\end{align*}
We can write
\begin{align*}
I^B(z)=2\pi\sum\limits_{m=1}^{\infty}M_m,
\end{align*}
where
\begin{align}\nonumber
M_m&~=\int_0^R\Big\{|\pa_rz_m|^2+|\pa_rz_{1-m}|^2+\frac1{r^2}(m^2|z_m|^2+(1-m)^2|z_{1-m}|^2)\\
&\qquad\qquad+\Big(\frac{a+b}{2ab}-\alpha\Big)(|z_m|^2+|z_{1-m}|^2)-\frac{a-b}{ab}Re(z_mz_{1-m})\Big\}rdr.\nonumber
\end{align}
Noticing that $m^2\ge 1,(1-m)^2\ge0$ for $m\ge0$, and using the following simple relations
\begin{eqnarray}
 Re(z_mz_{1-m})\le |z_mz_{1-m}|,\quad |\pa_rz_m|^2 \geq (\pa_r|z_m|)^2, \quad |\pa_rz_{1-m}|^2 \geq (\pa_r|z_{1-m}|)^2,\nonumber
\end{eqnarray}
we conclude that
\begin{align*}
M_m&\ge \int_0^R\Big((\pa_r|z_m|)^2+(\pa_r|z_{1-m}|)^2+\frac1{r^2}|z_{m}|^2\\
&\qquad\qquad+\Big(\frac{a+b}{2ab}-\alpha\Big)(|z_m|^2+|z_{1-m}|^2)-\frac{a-b}{ab}|z_m||z_{1-m}|\Big) rdr.
\end{align*}
Thus, we only need to show that for any $q_0 ,q_1\in C_c^\infty((0,\infty),\mathbb{R})$,
\begin{equation}\label{ibq}
\tilde{I}(q_0,q_1)\triangleq \int_0^{R}\Big\{(\pa_r q_0)^2+(\pa_r q_1)^2
+\frac{q_1^2}{r^2}+\Big(\frac{a+b}{2ab}-\alpha\Big)(q_0^2+q_1^2)-\frac{a-b}{ab}q_0q_1\Big\}rdr> 0.
\end{equation}
Let $\eta=q_1/u$ and $\zeta=q_0/v$. Then $\eta,\zeta\in C_c^{\infty}((0,R))$. When $|k|=1$,  it is straightforward to obtain
\begin{align}\label{ineq:B}\nonumber
\tilde{I}(q_0,q_1)& = \mathcal{A}(\zeta)+\mathcal{B}(\eta)-\int_0^{R}\Big\{\frac{3(a-b)v}{2abu}(u\eta)^2+\frac{(a-b)uv}{ab}\zeta\eta\Big\}rdr\\
&=\int_0^{R}\Big\{(u\eta')^2+(v\zeta')^2-\frac{a-b}{6ab}vu(3\eta+\zeta)^2\Big\} rdr,
\end{align}
which is non-negative since $v<0$ and $(a-b)u>0$.

\subsection{Non-negativity of $I^A$}
To estimate $I^A$, we need to calculate $\CL_Q^{-1}(E_i):E_j$ for $i,j=0,1,2$. For this, we let
\begin{align}
M_{ij}=\CL_Q(E_i):E_j=\int_{\BS}(E_i:\mm\mm)(E_j:\mm\mm)f_Q(\mm)d\mm
-(Q:E_i)(Q:E_j).
\end{align}
Easy calculation shows that $M_{02}=M_{12}=M_{20}=M_{21}=0$,
$M_{01}=M_{10}=c$, $M_{00}=d$, $M_{11}=e$, $M_{22}=h$. Here $c,d,e,h$ are
 defined in (\ref{def:abh})-(\ref{def:e}).
Then we have
\begin{equation}
  M=\left(
           \begin{array}{ccc}
             d &  c & 0 \\
             c  & e & 0 \\
             0 & 0 & h \\
           \end{array}
         \right),\quad\text{ and }
  M^{-1}=\left(
           \begin{array}{ccc}
             \frac{e}{de-c^2} &  -\frac{c}{de-c^2} & 0 \\
             -\frac{c}{de-c^2}  & \frac{d}{de-c^2} & 0 \\
             0 & 0 & \frac{1}{h} \\
           \end{array}
         \right).
\end{equation}
Therefore, we obtain
\begin{align}
I^A(w_0,w_1,w_2) =&  \int_0^{R}\int_0^{2\pi}\Big\{w_{0r}^2+w_{1r}^2+w_{2r}^2
+\frac{1}{r^2}[w_{0\varphi}^2+(w_2-w_{1\varphi})^2+(w_1+w_{2\varphi})^2]\\  \nonumber
&-\alpha(w_0^2+w_1^2+w_2^2)+\Big(\frac{ew_0^2}{de-c^2}-\frac{2cw_0w_1}{de-c^2}
+\frac{dw_1^2}{de-c^2}+\frac{w_2^2}{h}\Big)\Big\}
rdrd\varphi.
\end{align}

Expanding $w_i(r,\varphi)$ as
\begin{equation}
   w_i(r,\varphi)=\sum\limits_{n=0}^{\infty}\Big(
\mu_n^{(i)}(r)\cos{n\varphi}+\nu_n^{(i)}\sin{n\varphi}\Big),
\end{equation}
where all $\mu_n^{(i)}$ and $\nu_n^{(i)}$ belongs to $C_c^\infty(0,R)$, we can decompose $I^A(w_0,w_1,w_2)$ as
\begin{equation}
 I^A(w_0,w_1,w_2)=I_{0,01}^A+I_{0,2}^A+\sum\limits_{n=1}^{\infty} I_{n}^A,
\end{equation}
where
\begin{align}\label{iv0,01}
  I_{0,01}^A= &~2\pi \int_0^{R}\Big\{(\partial_r \mu_0^{(0)})^2
+(\mu_0^{(0)})^2\Big(-\alpha+\frac{e}{de-c^2}\Big)+(\partial_r
\mu_0^{(1)})^2\\
  \nonumber &~\qquad+(\mu_0^{(1)})^2\Big(-\alpha+\frac{d}{de-c^2}+\frac{1}{r^2}\Big)
-\frac{2c}{de-c^2}\mu_0^{(0)}\mu_0^{(1)} \Big\}rdr,\\
\label{iv0,2}
  I_{0,2}^A=&~2\pi\int_0^R \Big\{\big(\partial_r \mu_0^{(2)}\big)^2+
(\mu_0^{(2)})^2\Big(\frac{1}{r^2}-\alpha+\frac{1}{h}\Big)\Big\}rdr,\\
\label{ivn,012}
  I_{n}^A=&~\pi\int_0^{R}\bigg\{\sum\limits_{i=0}^2\Big[(\partial_r
\mu_n^{(i)})^2+(\partial_r \nu_n^{(i)})^2++\frac{n^2}{r^2}\big((\mu_n^{(i)})^2+(\nu_n^{(i)})^2\big)\Big]
+\frac{4n}{r^2}(\mu_n^1\nu_n^2-\mu_n^2\nu_n^1)\\\nonumber
&~\qquad+((\mu_n^{(0)})^2+(\nu_n^{(0)})^2)\Big(\frac{e}{de-c^2}-\alpha\Big)
+((\mu_n^{(1)})^2+(\nu_n^{(1)})^2)\Big(\frac{d}{de-c^2}-\alpha+\frac{1}{r^2}\Big)\\\nonumber
&~\qquad+((\mu_n^{(2)})^2+(\nu_n^{(2)})^2)\Big(\frac{1}{r^2}-\alpha+\frac{1}{h}\Big)
-\frac{2c}{de-c^2}(\mu_n^{(0)}\mu_n^{(1)}+\nu_n^{(0)}\nu_n^{(1)})\bigg\}rdr.
 \end{align}

\par
If we take $\mu_0^{(0)}=\sqrt{3}\nu,\, \mu_0^{(1)}=\mu$, then the first term $I_{0,01}^A$ is just the second variation of the
reduced energy $\mathcal{E}$ at the critical point $(u,v)$ (see \eqref{def:min-1d}). Thus, the
locally minimizing property of the solution $(u,v)$
implies the non-negativity of $I_{0,01}^A$.

For $I_{0,2}$, we recall the fact $u=fh$ from (\ref{relation:uf}) and  take $\eta=\f{\mu_0^{(2)}}{u}\in C_c^\infty(0,R)$, we obtain
\begin{align}\nonumber
I_{0,2}^A=&~\int_{0}^R \Big\{
(u\eta)_{r}^2+(\frac{f}{u}-\alpha+\frac{1}{r^2})(u\eta)^2\Big\}rdr\\
\nonumber =&~\mathcal{B}(\eta)=\int_{0}^\infty (u\eta_{r})^2rdr\geq 0.
\end{align}
\par
 Thus, we only need to prove
$I_n^A\geq 0$ for $n\ge 1$, which is the result of the following proposition.
\begin{proposition}
For any $\mu_0,\nu_0,\mu_1,\nu_1,\mu_2,\nu_2$, we have
\begin{align*}
  &I_{n}^A(\mu_0,\nu_0,\mu_1,\nu_1,\mu_2,\nu_2)\\\nonumber
  :=&~\pi\int_0^{R}\bigg\{\sum\limits_{i=0}^2\Big((\partial_r
\mu_i)^2+(\partial_r \nu_i)^2++\frac{n^2}{r^2}(\mu_i^2+\nu_i^2)\Big)
+\frac{4n}{r^2}(\mu_1\nu_2-\mu_2\nu_1)\\\nonumber
&~\qquad+(\mu_0^2+\nu_0^2)\Big(\frac{e}{de-c^2}-\alpha\Big)
+(\mu_1^2+\nu_1^2)\Big(\frac{d}{de-c^2}-\alpha+\frac{1}{r^2}\Big)\\\nonumber
&~\qquad+(\mu_2^2+\nu_2^2)\Big(\frac{1}{r^2}-\alpha+\frac{1}{h}\Big)
-\frac{2c}{de-c^2}(\mu_0\mu_1+\nu_0\nu_1)\bigg\}rdr\ge 0.
 \end{align*}
\end{proposition}

\begin{proof} From the fact that
\begin{align*}
2(\mu_1\nu_2-\mu_2\nu_1) \ge -(\mu_1^2+\nu_1^2+\mu_2^2+\nu_2^2),
\end{align*}
and $n\ge 1$, we get
\begin{align*}
&4n(\mu_1\nu_2-\mu_2\nu_1)+n^2(\mu_1^2+\nu_1^2+\mu_2^2+\nu_2^2)\ge
4(\mu_1\nu_2-\mu_2\nu_1)+(\mu_1^2+\nu_1^2+\mu_2^2+\nu_2^2).
\end{align*}
So, it suffices to consider the case of $n=1$.

On the other hand, we have
\begin{align*}
|\mu_0\mu_1+\nu_0\nu_1|\le \sqrt{\mu_0^2+\nu_0^2}\sqrt{\mu_1^2+\nu_1^2},\\
|\mu_1\nu_2-\mu_2\nu_1|\le \sqrt{\mu_2^2+\nu_2^2}\sqrt{\mu_1^2+\nu_1^2},
\end{align*}
and
$(\partial_r\mu)^2+(\partial_r\nu)^2\ge(\partial_r\sqrt{\mu^2+\nu^2})^2$.
Thus, we only need to prove that for
$\alpha_i=\pm\sqrt{\mu_{i}^2+\nu_{i}^2}$,
\begin{align*}
\tilde{I}^A_1(\alpha_0,
\alpha_1,\alpha_2)=&\int_0^{R}\Big\{(\partial_r\alpha_0)^2+(\partial_r\alpha_1)^2+(\partial_r\alpha_2)^2
-\frac{4}{r^2}\alpha_1\alpha_2+\frac{1}{r^2}(\alpha_0^2+\alpha_1^2+\alpha_2^2)\\
\nonumber&\qquad+\alpha_0^2\Big(\frac{e}{de-c^2}-\alpha\Big)+\alpha_1^2\Big(\frac{d}{de-c^2}-\alpha+\frac{1}{r^2}\Big)\\
\nonumber&\qquad+\alpha_2^2\Big(\frac{1}{r^2}-\alpha+\frac{1}{h}\Big)
-\frac{2c}{de-c^2}\alpha_0\alpha_1\Big\}rdr\geq0.
\end{align*}
Let $\xi=\alpha_0/v'$, $\eta=\alpha_1/u'$, $\zeta=r\alpha_2/u$. We have
\begin{align}\nonumber
&\tilde{I}^A_1(\alpha_0, \alpha_1,\alpha_2)\\\nonumber
&=\mathcal{C}(\xi)+\mathcal{D}(\eta)+\mathcal{G}(\zeta)\\
&\quad+\int_0^{R}\Big\{-\frac{4}{r^3}uu'\eta\zeta+\frac{1}{r^2}\Big((v'\xi)^2+(u'\eta)^2+(\frac{u\zeta}{r})^2\Big)
-\frac{2cu'v'}{de-c^2}\xi\eta\Big\}rdr\nonumber\\\label{ineq:A-1}
&=\int_0^{R}\Big\{(v'\xi')^2+(u'\eta')^2+\Big(\frac{u\zeta'}{r}\Big)^2+\frac{2uu'}{r^3}(\eta-\zeta)^2
+\frac{\sqrt{3}cu'v'}{de-c^2}\big(\eta-\frac{\xi}{\sqrt{3}}\big)^2\Big\}rdr\geq
0.\end{align}
This completes our proof.
\end{proof}

\section{Instability of radial symmetric solutions with degree $|k|>1$}\label{sec:instability}
In this section, we construct perturbations $w_3,w_4\in C_c^\infty(\BR\backslash\{0\})$
such that $I^B(w_3,w_4)$ is negative when $|k|>1$. This implies instability of radial symmetric solutions with degree $|k|>1$
which yields the conclusion of Theorem \ref{thm:instability}.

\par
Let $Q=u(r)F_1+v(r)F_2$ be the radial symmetric solution on $\BR$ with degree $k/2$($|k|>1$) and
$\epsilon$ be a very small parameter. Then there exists $R_0$ such that
for all $r>R_0$,
\begin{align*}
&(1-\epsilon)u(\infty)<u(r)<(1+\epsilon)u(\infty),\\
 &(1-\epsilon)v(\infty)<v(r)<(1+\epsilon)v(\infty).
\end{align*}
Then we take
\beq
w_3(r,\varphi)=q_0(r)+q_1(r)\cos{\varphi},\qquad w_4=q_1(r)\sin{\varphi}.
\eeq
where $q_0,q_1\in C_c^\infty(0,\infty)$. Direct computation gives that
\beq
I^B(w_3,w_4)=\tilde{I}(q_0,q_1):= \int_0^{\infty}\Big\{(\pa_r q_0)^2+(\pa_r q_1)^2
+\frac{q_1^2}{r^2}+\Big(\frac{a+b}{2ab}-\alpha\Big)(q_0^2+q_1^2)-\frac{a-b}{ab}q_0q_1\Big\}rdr> 0,
\eeq
where $\tilde{I}(q_0,q_1)$ is defined in (\ref{ibq}). Then we take
$q_1=u\eta$ and $q_0=-3v\eta$ where $\eta\in C_c^\infty(0,\infty)$. By
similar calculation as before, we obtain
\beq
\tilde{I}(q_0,q_1)=\int_0^{\infty}\Big\{(u^2+9v^2)(\eta')^2-\frac{k^2-1}{r^2}(u\eta)^2\Big\}
rdr,
\eeq
When $r>R_0$, $u$ and $v$ are almost constants and it is easy to find a
test function $\eta_0\in C_c^\infty(R_0,\infty)$ such that
\beq
\int_{R_0}^\infty\Big\{C_1 (\eta_0')^2-\f{C_2}{r^2}\eta_0^2\Big \}rdr<0,
\eeq
where $C_1, C_2$ are any positive constants. Thus Theorem \ref{thm:instability} is proved.


\section{Appendix: rotational gradient operator}\label{sec:appendix}
The rotational gradient operator on unit sphere $\BS$ is defined by
$$\mathcal{R}=\mm \times\nabla_{ \mm }$$
where $\nabla_\mm$ is the usual gradient operator with respect to the standard metric on $\BS$.
Under spherical coordinate $(\theta, \phi)$, it can be written as
\begin{align*}
\CR=&(-\sin\phi\mathbf{e}_1+\cos\phi\mathbf{e}_2)\partial_\theta
-(\cos\theta\cos\phi\mathbf{e}_1+\cos\theta\sin\phi\mathbf{e_2}-\sin\theta\mathbf{e}_3)
\frac{1}{\sin\theta}\partial_\phi\\
\triangleq&\ee_1{\CR}_1+\ee_2{\CR_2}+\ee_3{\CR_3}.
\end{align*}
We list some useful properties of $\CR$ here:
\begin{enumerate}
  \item
$  \int_{\BS}\mathcal{R} f_1 f_2 \ud\mm =-\int_{\BS}f_1\mathcal{R} f_2 \ud\mm.$
  \item $\CR_im_j=-\ve^{ijk}m_k$, where $\ve^{ijk}$ is the Levi-Civita symbol.
  \item $\CR(B:\mm\mm)=2\mm\times(B\cdot\mm)$ for constant matrix $B$.
\end{enumerate}
One may check these facts by direction calculation using the definition of $\CR$. We omit them here.
\section*{Acknowledgment}
Wei Wang is supported by NSF of China under Grant No. 11922118, 11871424 and 11501502. Zhiyuan Geng is supported by the Basque Government through the BERC 2022-2025 program and by the Spanish State Research Agency through BCAM Severo Ochoa excellence accreditation SEV-2017-0718 and through project PID2020-114189RB-I00 funded by Agencia Estatal de Investigaci\'{o}n (PID2020-114189RB-I00 / AEI / 10.13039/501100011033).


\begin{thebibliography}{50}











 \bibitem{BM} { J. M. Ball and A. Majumdar}, {\em Nematic liquid crystals: from Maier-Saupe to a continuum
 theory}, Mol. Cryst. Liq. Cryst., 525(2010),1-11.

\bibitem{BaP} P. Bauman, J. Park and D. Philips, {\it Analysis of nematic liquid crystals with disclination lines},
Arch. Ration. Mech. Anal., 205(2012), 795-826.

\bibitem{BP} P. Bauman and D. Phillips.
{\it Regularity and the behavior of eigenvalues for minimizers of a
  constrained {Q-tensor} energy for liquid crystals},
 Calculus of Variations and Partial Differential Equations,
  55.4(2016), 81.

\bibitem{C} G. Canevari, {\it Biaxiality in the asymptotic analysis of a 2-d
Landau-de Gennes model for liquid crystals}, ESAIM Control Optim. Calc. Var. 21(2015), 101-137.



\bibitem{dGP} P. de Gennes and J. Prost, {\it The Physics of Liquid Crystals}, second ed., Oxford University Press, Oxford, 1995.


\bibitem{DRSZ} G. Di Fratta, J. M. Robbins, V. Slastikov and A. Zarnescu, {\it Half-integer point defects in the Q-tensor
theory of nematic liquid crystals}, Journal of Nonlinear Science, 26(2016), 121-140.

\bibitem{E} J. Ericksen, {\it Liquid crystals with variable degree of orientation},
Arch. Ration. Mech. Anal., 113(1991), 97-120.

\bibitem{Evans2016partial}
L. C. Evans, O. Kneuss, and H. Tran.
{\it Partial regularity for minimizers of singular energy functionals,
  with application to liquid crystal models}
 Transactions of the American Mathematical Society,
  368.5(2016), 3389--3413.

\bibitem{FS} I. Fatkullin and V. Slastikov, {\it Critical points of the Onsager functional on a sphere},
Nonlinearity, 18(2005), 2565-80.

\bibitem{GM} E., Gartland, S. Mkaddem, {\it Instability of radial hedgehog configurations in nematic liquid crystals under Landau–de Gennes free-energy models}.
Physical Review E, 59(1999), 563-567.

\bibitem{GT} Z. Geng and J. Tong {\it  Regularity of minimizers of a tensor-valued variational obstacle problem in three dimensions} 	arXiv:1908.10889.


\bibitem{GWZZ} Z. Geng, W. Wang, P. Zhang and Z. Zhang, {\it Stability of half-degree point defect profiles for 2-D nematic liquid crystal}, Discrete $\&$ Continuous Dynamical Systems - A, 37.12(2017), 6227-6242.


\bibitem{GM} D. Golovaty and J. A. Montero, {\it On minimizers of a Landau-de Gennes energy
functional on planar domains}, Arch. Ration. Mech. Anal., 213.2(2014), 447-490.


\bibitem{HLW} J. Han, L. Yi, W. Wang, P. Zhang and Z. Zhang, {\it From microscopic theory to macroscopic theory: a systematic study on modeling for liquid crystals}, Arch. Rational Mech. Anal., 215(2015): 741-809.



\bibitem{HQZ} Y. Hu, Y. Qu and P. Zhang, {\it On the disclination lines of nematic liquid crystals}, Communications in Computational Physics 19.2(2016), 354-379.

\bibitem{INSZ0} R. Ignat, L. Nguyen, V. Slastikov, A. Zarnescu, {\it Uniqueness results for an ODE related to a generalized Ginzburg-Landau model for liquid crystals}, SIAM J. Math. Anal., 46(2014), 3390-3425.

\bibitem{INSZ1} R. Ignat, L. Nguyen, V. Slastikov and A. Zarnescu, {\it Stability of the melting hedgehog in the Landau-de Gennes theory of nematic liquid crystals},  Arch. Ration. Mech. Anal., 215(2015), 633-673.

\bibitem{INSZ2} R. Ignat, L. Nguyen, V. Slastikov and A. Zarnescu, {\it Instability of point defects in a two-dimensional nematic liquid crystal model},  Ann. I. H. Poincare-AN, 33.4(2016), 1131-1152.

\bibitem{INSZ3} R. Ignat, L. Nguyen, V. Slastikov and A. Zarnescu, {\it Stability of point defects of degree $\pm\frac12$ in a two-dimensional nematic liquid crystal model}, Calculus of Variations and Partial Differential Equations,  55.5(2016), 119.


\bibitem{La} X. Lamy, {\it Some properties of the nematic radial hedgehog in the Landau-de Gennes theory}, J. Math. Anal. Appl., 397(2013), 586-594.


\bibitem{LWZ} S. Li, W. Wang and P. Zhang,{\it Local well-posedness and small Deborah limit of a molecule-based Q-tensor system},  Disc. Conti. Dyn. Sys.-B, 20(2015), 2611-2655.

\bibitem{LL} F.-H. Lin and C. Liu, {\it Static and dynamic theories of liquid crystals}, J. Partial Differ. Equ., 14(2001), 289-330.


\bibitem{LZZ} H. Liu, H. Zhang and P. Zhang, {\it Axial symmetry and classification of stationary solutions of
Doi-Onsager equation on the sphere with Maier-Saupe potential}, Comm. Math. Sci., 3(2005), 201-218.


\bibitem{Ma} A. Majumdar, {\it The radial-hedgehog solution in Landau-de
Gennes' theory for nematic liquid crystals}, Euro. J. Appl. Math., 23(2012), 61-97.

\bibitem{MZ} A. Majumdar and A. Zarnescu, {\it Landau-de Gennes theory of nematic liquid crystals: the Oseen-Frank limit and beyond}, Arch. Ration. Mech. Anal., 196(2010), 227-280.


\bibitem{PWZZ} J. Park, W. Wang, P. Zhang and Z. Zhang, {\it On minimizers for the isotropic-nematic interface problem}, submitted.

\bibitem{RV}R. Rosso and E. G. Virga, {\it Metastable nematic hedgehogs}, J. Phys. A,  29(1996), 4247-4264.


\bibitem{WZZ1} W. Wang, P. Zhang and Z. Zhang, {\it The small Deborah number limit of the
Doi-Onsager equation to the Ericksen-Leslie equation}, Comm. Pure Appl. Math.,  68(2015), 1326-1398.

\bibitem{ZW} H. Zhou, H. Wang, {\it Stability of equilibria of nematic liquid crystalline polymers}. Acta Mathematica Scientia, 31(2011), 2289-2304.

\bibitem{ZWFW} H. Zhou, H. Wang, M. Forest, and Q. Wang, {\it A new proof on axisymmetric equilibria of a three-dimensional Smoluchowski equation},  Nonlinearity, 18(2005): 2815.


\end{thebibliography}
\end{document}